\documentclass[12pt,a4paper]{amsart}

\usepackage[utf8]{inputenc}
\usepackage{enumerate,a4,titletoc,graphicx}
\usepackage{bbm,amssymb,mathtools,mathrsfs}
\usepackage[usenames,dvipsnames]{xcolor}
\usepackage[pagebackref,colorlinks,linkcolor=BrickRed,citecolor=OliveGreen,urlcolor=black,hypertexnames=true]{hyperref}
\usepackage[margin=2.5cm]{geometry}

\DeclareMathOperator{\id}{id}
\DeclareMathOperator{\codim}{codim}
\DeclareMathOperator{\adj}{adj}
\DeclareMathOperator{\rk}{rk}
\DeclareMathOperator{\ran}{ran}
\DeclareMathOperator{\diag}{diag}
\DeclareMathOperator{\spa}{span}
\DeclareMathOperator{\GL}{GL}
\DeclareMathOperator{\SL}{SL}
\DeclareMathOperator{\UD}{UD}
\DeclareMathOperator{\GM}{GM}
\DeclareMathOperator{\tr}{tr}
\DeclareMathOperator{\dom}{dom}

\DeclareMathOperator{\opm}{M}
\DeclareMathOperator{\oph}{H}
\DeclareMathOperator{\hair}{hair}
\newcommand{\ve}{\varepsilon}
\newcommand{\N}{\mathbb{N}}
\newcommand{\R}{\mathbb{R}}
\newcommand{\C}{\mathbb{C}}

\newcommand{\cA}{\mathcal{A}}

\newcommand{\cC}{\mathcal{C}}
\newcommand{\cD}{\mathcal{D}}

\newcommand{\cM}{\mathcal{M}}

\newcommand{\cS}{\mathcal{S}}

\newcommand{\cZ}{\mathcal{Z}}

\newcommand{\cV}{\mathcal{V}}

\newcommand{\rr}{\mathbbm r}

\newcommand{\kk}{\mathbbm k}
\newcommand{\one}{{1\mathbbm a}}
\newcommand{\ti}{{\rm t}}

\newcommand{\fl}{\mathscr{Z}}
\newcommand{\flh}{\mathscr{Z}^{\operatorname{h}}}

\newcommand{\ulxi}{\boldsymbol{\omega}}
\newcommand{\ulyi}{\boldsymbol{\upsilon}}
\newcommand{\gx}{\omega}
\newcommand{\gy}{\upsilon}
\newcommand{\gX}{\Omega}
\newcommand{\gY}{\Upsilon}
\newcommand{\ulx}{\boldsymbol{x}}
\newcommand{\uly}{\boldsymbol{y}}

\newcommand*{\gl}[1]{\GL_{#1}(\kk)}
\newcommand*{\mat}[1]{\opm_{#1}(\kk)}
\newcommand*{\matc}[1]{\opm_{#1}(\C)}
\newcommand*{\herm}[1]{\oph_{#1}(\C)}

\newcommand{\Langle}{\mathop{<}\!}
\newcommand{\Rangle}{\!\mathop{>}}
\newcommand{\mx}{\Langle \ulx\Rangle}
\newcommand{\px}{\kk\!\mx}
\newcommand{\pxc}{\C\!\mx}
\newcommand{\pxy}{\kk\!\Langle \ulx,\uly\Rangle}

\makeatletter
\def\moverlay{\mathpalette\mov@rlay}
\def\mov@rlay#1#2{\leavevmode\vtop{
		\baselineskip\z@skip \lineskiplimit-\maxdimen
		\ialign{\hfil$#1##$\hfil\cr#2\crcr}}}
\makeatother

\newcommand{\plangle}{\moverlay{(\cr<}}
\newcommand{\prangle}{\moverlay{)\cr>}}
\newcommand{\rx}{\kk\plangle \ulx \prangle}

\newtheorem{thm}{Theorem}[section]
\newtheorem{lem}[thm]{Lemma}
\newtheorem{cor}[thm]{Corollary}
\newtheorem{prop}[thm]{Proposition}
\newtheorem{thmA}{Theorem}

\theoremstyle{definition}

\newtheorem{exa}[thm]{Example}

\theoremstyle{remark}
\newtheorem{rem}[thm]{Remark}

\numberwithin{equation}{section}

\begin{document}
	
\setcounter{tocdepth}{3}
\contentsmargin{2.55em} 
\dottedcontents{section}[3.8em]{}{2.3em}{.4pc} 
\dottedcontents{subsection}[6.1em]{}{3.2em}{.4pc}
\dottedcontents{subsubsection}[8.4em]{}{4.1em}{.4pc}

\makeatletter
\newcommand{\mycontentsbox}{%

	{\centerline{NOT FOR PUBLICATION}
		\addtolength{\parskip}{-2.3pt}
		\tableofcontents}}
\def\enddoc@text{\ifx\@empty\@translators \else\@settranslators\fi
	\ifx\@empty\addresses \else\@setaddresses\fi
	\newpage\mycontentsbox\newpage\printindex}
\makeatother

\setcounter{page}{1}

\title[Free loci and factorization of nc polynomials]{
	Geometry of free loci and factorization of noncommutative polynomials
	}

\author[J. W. Helton]{J. William Helton${}^1$}
\address{J. William Helton, Department of Mathematics, University of California San Diego}
\email{helton@math.ucsd.edu}
\thanks{${}^1$Research supported by the NSF grant DMS 1500835.
	The author was supported through the program ``Research in Pairs” (RiP) by the Mathematisches Forschungsinstitut Oberwolfach (MFO) in 2017.}

\author[I. Klep]{Igor Klep${}^2$}
\address{Igor Klep, Department of Mathematics, The University of Auckland}
\email{igor.klep@auckland.ac.nz}
\thanks{${}^2$Supported by the Marsden Fund Council of the Royal Society of New Zealand. Partially supported
	by the Slovenian Research Agency grants P1-0222, L1-6722, J1-8132. The author was supported through the program ``Research in Pairs” (RiP) by the Mathematisches Forschungsinstitut Oberwolfach (MFO) in 2017.}

\author[J. Vol\v{c}i\v{c}]{Jurij Vol\v{c}i\v{c}${}^3$}
\address{Jurij Vol\v{c}i\v{c}, Department of Mathematics\\
	Ben-Gurion University of the Negev \\ Israel}
\email{volcic@post.bgu.ac.il}
\thanks{${}^3$Supported by The University of Auckland Doctoral Scholarship.}

\date{\today}


\def\igor{\color{red}}
\def\jurij{\color{blue}}
\begin{abstract}
The free singularity locus of a noncommutative polynomial $f$ is defined to be the sequence of hypersurfaces
$\fl_n(f)= \left\{X\in\mat{n}^g\colon \det f(X)=0\right\}$.
The main theorem of this article shows that $f$ is irreducible if and only if
$\fl_n(f)$ is eventually irreducible. 
A key step in the proof is an irreducibility result for linear pencils. 
Arising from this is a free singularity locus Nullstellensatz for noncommutative polynomials.
Apart from consequences to factorization in a free algebra, 
the paper also discusses its applications to invariant subspaces in perturbation theory and linear matrix inequalities in real algebraic geometry.
\end{abstract}

\maketitle


\section{Introduction}

Algebraic sets, as zero sets of commutative polynomials are called, are basic objects in algebraic geometry and commutative algebra. 
One of the most fundamental results is 
Hilbert's Nullstellensatz, describing polynomials vanishing on an algebraic set. A simple special case of it is the following: if a polynomial $h$ vanishes on a hypersurface given as the zero set of an irreducible polynomial $f$, then $f$ divides $h$. Various far-reaching noncommutative versions of algebraic sets and corresponding Nullstellens\"atze have been introduced and studied by several authors \cite{Ami,vOV,RV,Scm}. Heavily reliant on these ideas and results are emerging areas of free real algebraic geometry \cite{dOHMP,HKN} and free analysis \cite{MS,KVV3,AM,KS}. In the free context there are several natural choices for the ``zero set'' of a noncommutative polynomial $f$. 
For instance, Amitsur proved a Nullstellensatz for the set of tuples of matrices $X$ satisfying $f(X)=0$ \cite{Ami}, and a conclusion for pairs $(X,v)$ of matrix tuples $X$ and nonzero vectors $v$ such that $f(X)v=0$ was given by Bergman \cite{HM}. In contrast with the successes in the preceding two setups, a Nullstellensatz-type analysis for the set of matrix tuples $X$ making $f(X)$ singular (not invertible), which we call the {\it free singularity locus} of $f$ (free locus for short), is much less developed. In this paper we rectify this. One of our main results connects free loci with factorization in free algebra \cite{Coh,BS,ARJ,BHL,Scr} in the sense of the special case of Hilbert's Nullstellensatz mentioned above.

A key intermediate step for studying components of free loci is an irreducibility theorem for determinants of linear pencils. To adapt this result for noncommutative polynomials we use a linearization process: given a noncommutative polynomial $f$, one can ``decouple'' products of variables using Schur complements to produce a linear matrix pencil $L$ with the same free locus as $f$. The most effective way of obtaining such an $L$ is via realizations originating in control theory \cite{BGM} and automata theory \cite{BR}.

The last objective of this paper is to derive consequences of our irreducibility theorem for hermitian pencils. They appear prominently across real algebraic geometry, see e.g. determinantal representations \cite{Bra,KPV}, the solution of the Lax conjecture \cite{HV,LPR} and the solution of the Kadison-Singer paving conjecture \cite{MSS}. Furthermore, hermitian pencils give rise to linear matrix inequalities (LMIs), the cornerstone of systems engineering \cite{SIG} and semidefinite optimization \cite{WSV}, and a principal research focus in free convexity \cite{EW,BPT,HKM,DDOSS}. Our irreducibility theorem enables us to analyze of the boundary of a free spectrahedron, also known as an LMI domain, associated to a  hermitian pencil $L$. We characterize the smooth points as those $X$ which make $\ker L(X)$ one-dimensional, and then prove the density of these points in the LMI boundary under mild hypotheses.

This paper is of possible interest to functional analysts (especially in free analysis), matrix theorists, those who study semidefinite programming, and to researchers in polynomial identities and representation theory. For a different audience the paper could be written using terminology from invariant theory as was explained to us by \v{S}pela \v{S}penko and Claudio Procesi; see \cite[Appendix A]{KV} for a discussion.

\subsection*{Main results}

Let $\kk$ be an algebraically closed field of characteristic $0$, $\ulx=(x_1,\dots,x_g)$ a tuple of freely noncommuting variables and $\px$ the free $\kk$-algebra over $\ulx$. To a noncommutative polynomial $f\in\px$ we assign its {\it free (singularity) locus}
$$\fl(f)=\bigcup_{n\in\N} \fl_n(f),\qquad \text{where}\quad
\fl_n(f)=\left\{X\in\mat{n}^g\colon \det f(X)=0\right\}.$$
If a noncommutative polynomial $f$ factors as a product of two nonconstant polynomials, $f=f_1f_2$, then $\fl(f)=\fl(f_1)\cup\fl(f_2)$. It is easy to see that the free locus of a nonconstant polynomial is nonempty, so the hypersurface $\fl_n(f)$ has at least two components for large enough $n\in\N$. Our main result is the converse to this simple observation.

\begin{thmA}\label{ta:1}
Let $f\in\px$ satisfy $f(\kk^g)\neq\{0\}$. Then $f$ is irreducible if and only if 
there exists $n_0\in\N$ such that $\fl_n(f)\subset \mat{n}^g$ is an irreducible hypersurface for all $n\ge n_0$.
\end{thmA}

Note that in general we cannot take $n_0=1$ in Theorem \ref{ta:1}; for example, the noncommutative polynomial $f=(1-x_1)^2-x_2^2$ is irreducible, but $f(\omega_1,\omega_2)=(1-\omega_1-\omega_2)(1-\omega_1+\omega_2)$ for commuting indeterminates $\omega_1,\omega_2$. Theorem \ref{ta:1} is a corollary of the more general Theorem \ref{t:irrpoly} below, which certifies the inclusion of free loci of noncommutative polynomials. For example, consider irreducible polynomials $f_1=1+x_1x_2$ and $f_2=1+x_2x_1$. It is well-known that for square matrices $A$ and $B$ of equal size, the eigenvalues of $AB$ and $BA$ coincide, so $\fl(f_1)=\fl(f_2)$ holds. On the other hand, this equality also follows from 
$$
\begin{pmatrix}1 & x_1 \\ 0 & 1\end{pmatrix}
\begin{pmatrix}1+x_1x_2 & 0 \\ 0 & 1\end{pmatrix}
\begin{pmatrix}1 & 0 \\ -x_2 & 1\end{pmatrix}
=
\begin{pmatrix}0 & 1 \\ -1 & -x_2\end{pmatrix}
\begin{pmatrix}1+x_2x_1 & 0 \\ 0 & 1\end{pmatrix}
\begin{pmatrix}0 & -1 \\ 1 & x_1\end{pmatrix}.
$$
Such an algebraic condition, called {\it stable associativity} of (irreducible) polynomials $f_1$ and $f_2$, is 
necessary for $\fl(f_1)=\fl(f_2)$ to hold. 
More precisely, we obtain the following  free locus Nullstellensatz for polynomials:

\begin{thmA}\label{ta:null}
Let $f_1,f_2\in\px$ satisfy $f_i(\kk^g)\neq\{0\}$.
Then $\fl(f_1)\subseteq\fl(f_2)$ if and only if 
each irreducible factor of $f_1$ is 
(up to stable associativity)
an irreducible factor of $f_2$.
\end{thmA}

Theorems \ref{ta:1} and \ref{ta:null} are special cases of 
Theorem \ref{t:irrpoly} below,
and both hold for matrix polynomials. 
The proof of Theorem \ref{t:irrpoly} consists of two ingredients. The first one is Cohn's factorization theory for semifirs \cite{Coh}, which deals with the ring-theoretic side of factorization. The second one is Theorem \ref{ta:2}, an irreducibility result about evaluations of linear matrix pencils, which we discuss next.

Given a {\it monic linear pencil} $L=I_d-A_1x_1-\cdots-A_gx_g$ with $A_j\in\mat{d}$, its evaluation at $X\in\mat{n}^g$ is defined as
$$L(X)=I_{dn}-A_1\otimes X_1-\cdots-A_g\otimes X_g\in\mat{dn},$$
where $\otimes$ is the Kronecker product. We say that a monic pencil is {\it irreducible} if its coefficients generate the whole matrix algebra, or equivalently, they do not admit a non-trivial common invariant subspace. For $k=1,\dots, g$ let $\gX^{(n)}_k$ be an $n\times n$ generic matrix, i.e., a matrix of $n^2$ independent commuting variables $\gx_{kij}$ for $1\le i,j\le n$.

\begin{thmA}\label{ta:2}
Let $L$ be an irreducible monic pencil. Then there exists $n_0\in\N$ such that $\det L(\gX^{(n)}_1,\dots,\gX^{(n)}_g)$ is an irreducible polynomial for all $n\ge n_0$.
\end{thmA}

See Theorem \ref{t:main} for the proof, which combines in a novel way invariant theory for the action of $\gl{n}$ on $\mat{n}^g$ by simultaneous conjugation and free analysis techniques for transitioning between different sizes of generic matrices. Given a linear pencil $L$ we define its free locus analogously as in the preceding setting of noncommutative polynomials:
$$\fl(L)=\bigcup_{n\in\N} \fl_n(L),\qquad \text{where}\quad
\fl_n(L)=\left\{X\in\mat{n}^g\colon \det L(X)=0\right\}.$$
Theorem \ref{ta:1} is then deduced from Theorem \ref{ta:2} using a linearization process \cite[Section 5.8]{Coh}, which to every noncommutative polynomial $f$ assigns a linear pencil $L$ with $\fl(f)=\fl(L)$. 

In the second part of the paper we thus turn our attention to monic pencils $L$ with polynomial free loci, i.e., $\fl(L)=\fl(f)$ for some $f\in\px$. We apply noncommutative Fornasini-Marchesini state space realizations \cite{BV,BGM} to prove that the coefficients of such pencils $L$ are of the form $N_j+E_j$, where $N_j$ are jointly nilpotent matrices and $E_j$ are rank-one matrices with coinciding kernels (Corollary \ref{c:polyloc}). By connecting this result with Theorem \ref{ta:1} and minimal factorizations in the sense of realization theory we obtain a curious statement about invariant subspaces. Given a tuple of matrices $A=(A_1,\dots,A_g)$ and $A'=(A_1',\dots,A_g')$ we say that $A'$ is a {\it non-degenerate right rank-one perturbation} of $A$ if $A'_j-A_j=b_jc^\ti$ for some vectors $b_j$ and $c$ such that $\{b_1,\dots,b_g\}$ is not contained in a proper invariant subspace for $A_1,\dots,A_g$ and $c$ does not lie in a proper invariant subspace for $A_1^\ti,\dots,A_g^\ti$. The following is a consequence of Theorem \ref{t:pert}.

\begin{thmA}\label{ta:3}
Let $A$ be a tuple of jointly nilpotent matrices. If a non-degenerate right rank-one
perturbation of $A$ has an invariant subspace $\cS$, then there exists a subspace $\cS^\times$ that is complementary to $\cS$ and invariant under $A$.
\end{thmA}

Lastly we consider consequences of Theorem \ref{ta:2} in real algebraic geometry. A monic pencil $L$ is hermitian if its coefficients are hermitian matrices. Its {\it free spectrahedron} or {\it LMI domain} is defined as
$$\cD(L)=\bigcup_{n\in\N} \cD_n(L),\qquad \text{where}\quad
\cD_n(L)=\left\{X\in\herm{n}^g\colon L(X)\succeq0\right\}.$$
Here $\herm{n}$ denotes $n\times n$ hermitian matrices and $M\succeq0$ means that $M$ is positive semidefinite. Then $\cD_n(L)\subseteq \herm{n}^g$ is a convex set and its boundary is contained in $\fl_n(L)$. For $n\in\N$ denote
$$\partial^1\cD_n(L)=\left\{ X\in\cD_n(L)\colon \dim\ker L(X)=1\right\}.$$
The points in $\partial^1\cD_n(L)$ are often precisely the smooth points of the boundary of $\cD_n(L)$ \cite{Ren}. The unique-null-vector property makes them vital for optimization \cite{Ren}, Positivstellens\"atze in free real algebraic geometry \cite{HKN} and the study of free analytic maps between LMI domains \cite{AHKM}. Unfortunately, given a fixed $n\in\N$ it can happen that $\partial^1\cD_n(L)=\emptyset$ even for irreducible hermitian pencils, as a consequence of the failure of Kippenhahn's conjecture \cite{Laf}. However, we show that this cannot happen for every $n\in\N$ if $L$ is {\it LMI-minimal}, i.e., of minimal size among all hermitian pencils whose LMI domain equals $\cD(L)$.

\begin{thmA}\label{ta:4}
Let $L$ be an LMI-minimal hermitian pencil. Then there exists $n_0\in\N$ such that $\partial^1\cD_n(L)$ is Zariski dense in $\fl_n(L)$ for all $n\ge n_0$.
\end{thmA}

Theorem \ref{ta:4} can be viewed as a quantitative solution of the quantum Kippenhahn conjecture and is proved as Corollary \ref{c:pd-min} below.

\subsection*{Acknowledgments}
We wish to thank Claudio Procesi, who read an early version of the manuscript, for his comments.
The second named author also thanks \v Spela \v Spenko for patiently sharing her expertise in invariant theory.

\section{Preliminaries}\label{s2}

In this section we gather results about free loci and polynomial invariants for the general linear group acting on matrix tuples by conjugation that will be used in the sequel.

Let $\kk$ be algebraically closed field of characteristic 0. Throughout the paper we use the following convention. If $\mathscr{S}(n)$ is a statement depending on $n\in\N$, then ``$\mathscr{S}(n)$ holds for large $n$'' means ``there exists $n_0\in\N$ such that $\mathscr{S}(n)$ holds for all $n\ge n_0$''.

\subsection{Free loci of matrix pencils}

For $g\in\N$ let $\ulx=(x_1,\dots,x_g)$ be a tuple of freely noncommuting variables. If $A_1,\dots,A_g\in\mat{d}$, then
$$L=I_d-\sum_{j=1}^g A_jx_j$$
is a {\bf monic linear pencil of size $d$}. For $X\in\mat{n}^g$ let
$$L(X)=I_{dn}-\sum_j A_{j=1}^g\otimes X_j \in\mat{dn},$$
where $\otimes$ denotes the Kronecker product. Let $\cM^g=\bigcup_{n\in\N}\mat{n}^g$. The set
$$\fl(L)=\bigcup_{n\in\N}\fl_n(L)\subset\cM^g, \qquad \text{where}\quad
\fl_n(L)=\left\{X\in\mat{n}^g\colon \det L(X)=0\right\},$$
is the {\bf free locus of $L$}. We review some terminology and facts about free loci from \cite{KV} that will be frequently used throughout the paper.

\begin{enumerate}
	\item $\fl(L)=\emptyset$ if and only if $A_1,\dots,A_g$ are jointly nilpotent by \cite[Corollary 3.4]{KV}.
	\item If $A_1,\dots,A_g$ generate $\mat{d}$ as a $\kk$-algebra, we say that $L$ is an {\bf irreducible} pencil. If $L_1$ and $L_2$ are irreducible and $\fl(L_1)\subseteq \fl(L_2)$, then $\fl(L_1)= \fl(L_2)$ (follows from \cite[Theorem 3.6]{KV} because a surjective homomorphism from a simple algebra is an isomorphism) and moreover $L_1$ and $L_2$ are similar, i.e., they differ only by a basis change on $\C^d$, by \cite[Theorem 3.11]{KV}.
	\item A free locus is {\bf irreducible} if it is not a union of smaller free loci. By \cite[Proposition 3.12]{KV}, a free locus is irreducible if and only if it is a free locus of some irreducible pencil.
	\item By applying Burnside's theorem \cite[Corollary 5.23]{Bre} on the existence of invariant subspaces of the $\kk$-algebra generated by $A_1,\dots,A_g$ it follows that every monic pencil $L$ is similar to a pencil of the form
	\begin{equation}\label{e:05}
	\begin{pmatrix}
	L_1 & \star & \cdots & \star \\
	0 & \ddots & \ddots & \vdots \\
	\vdots & \ddots & \ddots & \star \\
	0 & \cdots & 0 & L_\ell
	\end{pmatrix},
	\end{equation}
	where for every $k$, $L_k=I$ or $L_k$ is an irreducible pencil.
	\item Finally, we say that $L\neq I$ is {\bf FL-minimal} (free locus minimal) if it is of minimal size among pencils $L'$ with $\fl(L')=\fl(L)$. An $L$ of the form \eqref{e:05} is FL-minimal if and only if the $L_k$ are pairwise non-similar irreducible pencils. Furthermore, 
the diagonal blocks of an FL-minimal pencil
of the form \eqref{e:05} are unique up to a basis change by \cite[Theorem 3.11]{KV}.
\end{enumerate}

\subsection{Simultaneous conjugation of matrices}

Consider the action of $\gl{n}$ on the space $\mat{n}^g$ given by
\begin{equation}\label{e:00}
\sigma\cdot (X_1,\dots,X_g) = (\sigma X_1\sigma^{-1},\dots,\sigma X_g\sigma^{-1}),\qquad \sigma\in\gl{n}.
\end{equation}
The coordinate ring of $\mat{n}^g$ is the polynomial ring $\kk[\ulxi]$ in $gn^2$ commuting variables $\gx_{j\imath\jmath}$ for $1\le j \le g$ and $1\le \imath,\jmath\le n$. For $1\le j\le g$ let $\gX^{(n)}_j=(\gx_{j\imath\jmath})_{\imath,\jmath}$ be an $n\times n$ generic matrix and write  $\gX^{(n)}=(\gX^{(n)}_1,\dots,\gX^{(n)}_g)$. Viewing $p\in\kk[\ulxi]$ as a polynomial in the entries of $\gX^{(n)}$, the action \eqref{e:00} induces an action of $\gl{n}$ on $\kk[\ulxi]$ defined by
\begin{equation}\label{e:01}
\sigma \cdot p=p(\sigma^{-1}\cdot\gX^{(n)}), \qquad \sigma\in\gl{n}, \ p\in\kk[\ulxi].
\end{equation}
The subring of $\kk[\ulxi]$ of invariants for this action is denoted $\kk[\ulxi]^{\gl{n}}$. By \cite[Theorem 1.3]{Pro} it is generated by $\tr(w(\gX^{(n)}))$ for $w\in\mx\setminus\{1\}$. Here $\tr$ denotes the usual trace, and $\mx$ is the free monoid generated by the freely noncommuting variables $x_1,\dots,x_g$. The elements of $\kk[\ulxi]^{\gl{n}}$ are therefore called {\it pure trace polynomials}. For example, $2-\tr(\gX_1^{(n)})\tr(\gX_2^{(n)})+\tr(\gX_1^{(n)2}\gX_2^{(n)3}\gX_1^{(n)}\gX_2^{(n)})$ is a pure trace polynomial.

In the next section we require the following two lemmas. While they are known to specialists, we provide their proofs to keep the presentation self-contained.

\begin{lem}\label{l:invfac}
Every factor of a $\gl{n}$-invariant polynomial is $\gl{n}$-invariant.
\end{lem}

\begin{proof}
Let $q$ be an irreducible factor of $p\in\kk[\ulxi]^{\gl{n}}$. Let $\sigma\in\gl{n}$; since $\sigma\cdot p=p$ and $\sigma\cdot q$ is irreducible, we see that $\sigma\cdot q$ is again an irreducible factor of $p$. Let $\cS$ be the set of all nonzero scalar multiples of irreducible factors of $p$. Then $\cS$ with the Zariski topology is homeomorphic to a disjoint union of copies of $\kk^*=\kk\setminus\{0\}$. The algebraic group $\gl{n}$ is irreducible, and the map
$$f:\gl{n}\to \cS,\qquad \sigma\mapsto \sigma \cdot q$$
is regular and $f(\id)=q$. Therefore the image of $f$ lies in the irreducible component of $\cS$ containing $q$, so for every $\sigma\in\gl{n}$ we have $\sigma \cdot q=\lambda_\sigma q$ for some $\lambda_\sigma\in\kk^*$. Next we observe that the map
$$\lambda\colon\gl{n}\to \kk^*,\qquad \sigma\mapsto \lambda_{\sigma}$$
is a group homomorphism. Note that $\lambda(\kk^* I)=\{1\}$ by the nature of our action; also $\lambda(\SL_n(\kk))=\{1\}$ because $\SL_n(\kk)$ is generated by multiplicative commutators. Since $\gl{n}$ is a semidirect product of $\SL_n(\kk)$ and $\kk^*$, we conclude that $\lambda$ is constantly equal to 1, so $q$ is a $\gl{n}$-invariant.
\end{proof}

\begin{lem}\label{l:triang}
Let $X\in \mat{n}^g$ and $X'\in(\kk^{n'\times n})^g$. If $p$ is $\gl{n+n'}$-invariant, then
$$p\begin{pmatrix} 0 & X' \\ 0 & X \end{pmatrix}=p\begin{pmatrix} 0 & 0\\ 0 & X \end{pmatrix}.$$ 
\end{lem}

\begin{proof}
For every $w\in\mx\setminus\{1\}$ we clearly have
$$w\begin{pmatrix} 0 & X' \\ 0 & X \end{pmatrix}=\begin{pmatrix} 0 & Y_w\\ 0 & w(X) \end{pmatrix}$$
for some $Y_w\in\kk^{n'\times n}$ and hence
$$\tr\left(w\begin{pmatrix} 0 & X' \\ 0 & X \end{pmatrix}\right)
=\tr \begin{pmatrix} 0 & Y_w\\ 0 & w(X) \end{pmatrix}
=\tr(w(X))
=\tr\left(w\begin{pmatrix} 0 & 0 \\ 0 & X \end{pmatrix}\right).$$
The statement now follows because $\kk[\ulxi]^{\gl{n+n'}}$ is generated by $\tr(w(\gX^{(n+n')}))$.
\end{proof}

\section{Determinant of an irreducible pencil}\label{s3}

Let $L$ be a monic pencil and let $\gX^{(n)}$ be a $g$-tuple of $n\times n$ generic matrices. In Subsection \ref{ss32} we prove
our first main result,
 Theorem \ref{t:main}. which states that for an irreducible pencil $L$, the commutative polynomial $\det L(\gX^{(n)})$ is irreducible for large $n$. The consequences for the locus $\fl_n(L)$ are given in Subsection \ref{ss33}.

\subsection{Degree growth}\label{ss31}

Let $\GM_n(g)\subseteq \opm_n(\kk[\ulxi])$ be the ring of $n\times n$ generic matrices, i.e., the unital $\kk$-algebra generated by $\gX_1^{(n)},\dots,\gX_g^{(n)}$ \cite[Section 5]{For}. Furthermore, let $\UD_n(g)\subseteq \opm_n(\kk(\ulxi))$ be the universal division algebra of degree $n$, which is the ring of central quotients of $\GM_n(g)$.

\begin{lem}\label{l:deg}
Let $L$ be a monic pencil of size $d$ and
$$f_n=\det L\left(\gX^{(n)}\right)\in \kk[\ulxi]$$
for $n\in\N$. Then there exists $d_L\le d$ such that $\deg f_n=d_L n$ for all $n\ge d^2-1$.
\end{lem}

\begin{proof}
The case $d=1$ is clear, so assume $d\ge2$. Let $\Lambda=I-L$ and define a non-monic matrix pencil $\tilde{L}$ of size $d^2$ as
$$\tilde{L}=\begin{pmatrix}
0 & -\Lambda & 0 & \cdots & \\
\vdots & I & \ddots & & \vdots \\
& & \ddots & -\Lambda & 0\\
 & & & I & -\Lambda \\
\Lambda & & \cdots & 0 & I 
\end{pmatrix}.$$
It is easy to check that
\begin{equation}\label{e:02}
\begin{pmatrix}
I & \Lambda & \cdots & \Lambda^{d-1} \\
 & \ddots & \ddots & \vdots \\
 & & \ddots & \Lambda \\
 & & & I
\end{pmatrix}
\cdot \tilde{L} =
\begin{pmatrix}
\Lambda^d & & & \\
\Lambda^{d-1} & I & & \\
\vdots & & \ddots &\\
\Lambda &  &  & I
\end{pmatrix}.
\end{equation}
Next we make a few simple observations. Firstly, if $A$ is an $a\times a$ matrix over a (commutative) field $F$, then the degree of the univariate polynomial $\det(I-tA)\in F[t]$ equals the rank of $A^a$ over $F$. Secondly, if $B$ is a $b\times b$ matrix over a skew field, then the rank of $B^{b_1}$ equals the rank of $B^b$ for every $b_1\ge b$; here the rank of a matrix over a skew field equals the dimension of its range as a linear operator over the skew field. Finally, if $C$ is a $c\times c$ matrix over $\UD_n(g)$ of rank $r$, then it is equivalent to $I^{\oplus r}\oplus 0^{\oplus(c-r)}$ over $\UD_n(g)$, so $C$ is of rank $rn$ as a $cn\times cn$ matrix over $\kk(\ulxi)$.

Let $n\in\N$ be arbitrary. Viewing $\Lambda(\gX^{(n)})$ as a $d\times d$ matrix over $\UD_n(g)$, the preceding observations and \eqref{e:02} imply
\begin{align*}
\deg f_n
& = \deg_t \left(I-t\Lambda\left(\gX^{(n)}\right)\right) \\
& = \rk_{\kk(\ulxi)} \Lambda\left(\gX^{(n)}\right)^{dn} \\
& = \frac{1}{n}\rk_{\UD_n(g)} \Lambda\left(\gX^{(n)}\right)^{dn} \\
& = \frac{1}{n}\rk_{\UD_n(g)} \Lambda\left(\gX^{(n)}\right)^d \\
& = \rk_{\kk(\ulxi)} \Lambda\left(\gX^{(n)}\right)^d \\
& = \rk_{\kk(\ulxi)} \tilde{L}\left(\gX^{(n)}\right) - (d-1)d n.
\end{align*}
By the proof of \cite[Proposition 2.10]{DM}, there is $d'\in\N$ such that $\rk_{\kk(\ulxi)} \tilde{L}(\gX^{(n)})=d' n$ for all $n\ge d^2-1$. Hence $d_L=d'-(d-1)d$.
\end{proof}

\begin{exa}\label{ex0}
Let
$$A_1=\begin{pmatrix}1 & 0 & 0 \\ 1 & 1 & 0 \\ 0 & 0 & 0\end{pmatrix}, \qquad
A_2=\begin{pmatrix}0 & 0 & 1 \\ 0 & 0 & 1 \\ 0 & 1 & 0\end{pmatrix}.$$
One can check that $A_1$ and $A_2$ generate $\mat{3}$. If $L=I-A_1x_1-A_2x_2$, then using Schur complements we see that
$$\det L\left(\gX^{(n)}\right)=
\det\begin{pmatrix}
I-\gX_1^{(n)} & -(\gX_2^{(n)})^2 \\ -\gX_1^{(n)} & I-\gX_1^{(n)}-(\gX_2^{(n)})^2
\end{pmatrix}=
\det \left((I-\gX_1^{(n)})^2-(\gX_2^{(n)})^2\right)$$
is of degree $2n$ for every $n\in\N$. Therefore $d_L=2<3$.
\end{exa}

\subsection{Eventual irreducibility}\label{ss32}

Let $\gY^{(n)}$ be an $n\times n$ generic matrix whose entries are independent of the entries in $\gX^{(n)}_j$ (that is, we introduce $n^2$ new variables to form $\gY^{(n)}$). The following lemma is a key technical tool for proving Theorem \ref{t:main} below.

\begin{lem}\label{l:main}
Let $L=I-\sum_{j=1}^g A_jx_j$ be of size $d$ and $n_0\ge d^2-1$. Fix $1\le j',j''\le g$ and assume that
$$\det\left(L\left(\gX^{(n)}\right)-A_{j'}A_{j''} \otimes \gY^{(n)}\right)$$
is an irreducible polynomial (in $(g+1)n^2$ variables) for every $n\ge n_0$. Then $\det L\left(\gX^{(n)}\right)$ is an irreducible polynomial (in $gn^2$ variables) for every $n\ge 2n_0$.
\end{lem}

\begin{proof}
For $n\ge n_0$ denote
$$f_n=\det L\left(\gX^{(n)}\right), \qquad
\hat{f}_n=\det\left(L\left(\gX^{(n)}\right)-A_{j'}A_{j''} \otimes \gY^{(n)}\right).$$
Suppose that $f_{2n}=pq$ for some nonconstant polynomials $p$ and $q$. By Lemma \ref{l:invfac}, $p$ and $q$ are $\gl{2n}$-invariant. Let $\gX'^{(n)}_j, \gX''^{(n)}_j$ be independent $n\times n$ generic matrices and denote
$$\tilde{p}=\tilde{p}(\gX'^{(n)}, \gX''^{(n)}):=p(\gX'^{(n)}\oplus \gX''^{(n)}),\qquad
\tilde{q}=\tilde{q}(\gX'^{(n)}, \gX''^{(n)}):=q(\gX'^{(n)}\oplus \gX''^{(n)}).$$
Note that
\begin{equation}\label{e:11}\tilde{p}(\gX^{(n)}, 0)=\tilde{p}(0, \gX^{(n)}),\qquad
\tilde{q}(\gX^{(n)}, 0)=\tilde{q}(0, \gX^{(n)})
\end{equation}
since $p$ and $q$ are $\gl{2n}$-invariant. Consider
\begin{equation}\label{e:12}
f_n(\gX'^{(n)})f_n(\gX''^{(n)})
=f_{2n}(\gX'^{(n)}\oplus \gX''^{(n)})
=\tilde{p}\tilde{q}.
\end{equation}
By Lemma \ref{l:deg}, the left-hand side of \eqref{e:12} has degree $2d_Ln$. Since $p$ and $q$ have degree strictly less than $2d_Ln$, we conclude that $\tilde{p}$ and $\tilde{q}$ are nonconstant. Moreover, since the left-hand side of \eqref{e:12} is a product of two polynomials in disjoint sets of variables, we conclude that
$$\tilde{p}=\tilde{p}_1(\gX'^{(n)})\tilde{p}_2(\gX''^{(n)}), \qquad
\tilde{q}=\tilde{q}_1(\gX'^{(n)})\tilde{q}_2(\gX''^{(n)})$$
for some polynomials $\tilde{p}_1,\tilde{p}_2,\tilde{q}_1,\tilde{q}_2$. If $\tilde{p}_1$ were constant, then $\tilde{p}_2$ would be constant by \eqref{e:11}, contradicting that $\tilde{p}$ is nonconstant. Hence we conclude that $\tilde{p}_1,\tilde{p}_2,\tilde{q}_1,\tilde{q}_2$ are nonconstant and consequently
$$p\begin{pmatrix}0 & 0 \\ 0 & \gX^{(n)}\end{pmatrix}= \tilde{p}(0,\gX^{(n)}), \qquad
q\begin{pmatrix}0 & 0 \\ 0 & \gX^{(n)}\end{pmatrix}= \tilde{q}(0,\gX^{(n)})$$
are nonconstant.

Since
$$\det\left( L\left(\gX^{(n)}\right)-A_{j'}A_{j''} \otimes \gY^{(n)}\right)
=\det\begin{pmatrix}
I & A_{j''}\otimes I \\ A_{j'}\otimes \gY^{(n)} & L\left(\gX^{(n)}\right)
\end{pmatrix},$$
we see that
$$\hat{f}_n=f_{2n}(\cZ),$$
where
$$\cZ_j=\begin{pmatrix}0 & 0 \\ 0 & \gX^{(n)}_j\end{pmatrix}$$
for $j\notin\{j',j''\}$, and
$$\cZ_{j'}=\begin{pmatrix}0 & 0 \\ \gY^{(n)} & \gX^{(n)}_{j'}\end{pmatrix}, \qquad
\cZ_{j''}=\begin{pmatrix}0 & I \\ 0 & \gX^{(n)}_{j''}\end{pmatrix}$$
if $j'\neq j''$ and
$$\cZ_{j'}=\begin{pmatrix}0 & I \\ \gY^{(n)} & \gX^{(n)}_{j'}\end{pmatrix}$$
if $j'=j''$. Since $\cZ|_{\gY^{(n)}=0}$ is a tuple of block upper triangular matrices, Lemma \ref{l:triang} implies
$$p(\cZ)|_{\gY^{(n)}=0} = p\begin{pmatrix}0 & 0 \\ 0 & \gX^{(n)}\end{pmatrix}, \qquad
q(\cZ)|_{\gY^{(n)}=0} = q\begin{pmatrix}0 & 0 \\ 0 & \gX^{(n)}\end{pmatrix}.$$
In particular, $p(\cZ)$ and $q(\cZ)$ are nonconstant and
$$\hat{f}_n=p(\cZ)q(\cZ),$$
a contradiction.

Hence we have proven the statement for every even $n\ge 2n_0$. If $n\ge 2n_0$ is odd, then $n-1\ge 2n_0$. If $f_n=pq$, then irreducibility of $f_{n-1}$ implies that $p(0\oplus\gX^{(n-1)})=1$ and $q(0\oplus\gX^{(n-1)})=f_{n-1}$ (or vice versa). Therefore $q$ is of degree at least $(n-1)d_L$, so $1-p$ is of degree at most $d_L$, is a pure trace identity for $\mat{n-1}$ but not for $\mat{n}$. However, this cannot happen by \cite[Theorem 4.5]{Pro} since $n-1\ge 2n_0\ge d_L$.
\end{proof}

We are now ready to prove the first of our main results, which was announced in \cite{Vol2}. After being informed of Theorem \ref{t:main}, Kriel \cite{Kri} independently proved it for hermitian pencils.

\begin{thm}\label{t:main}
If $L$ is an irreducible pencil, then $\det L\left(\gX^{(n)}\right)$ is an irreducible polynomial for large $n$.
\end{thm}

\begin{proof}
Let $L=I-\sum_jA_jx_j$; then $A_1,\dots,A_g\in\mat{d}$ generate $\mat{d}$ as a $\kk$-algebra by assumption. If
$$A_g=\sum_{j=1}^{g-1}\alpha_jA_j,$$
then $A_1,\dots,A_{g-1}$ generate $\mat{d}$ and
$$\det L\left(\gX^{(n)}\right)
=\det\left(I-\sum_{j=1}^{g-1} A_j\otimes \left(\gX^{(n)}_j+\alpha_j\gX^{(n)}_g\right)\right),$$
so irreducibility of $\det\left(I-\sum_{j<g} A_j\otimes \gX^{(n)}_j\right)$ implies irreducibility of $\det L\left(\gX^{(n)}\right)$. Therefore we can without loss of generality assume that $A_1,\dots,A_g$ are linearly independent.

Let $w_{\imath\jmath}\in\mx$ for $1\le \imath,\jmath\le d$ be such that
\begin{equation}\label{e:120}
\{x_1,\dots,x_g\}\subseteq\{w_{\imath\jmath}\colon 1\le \imath,\jmath\le d\}, \qquad
\spa\left\{w_{\imath\jmath}(A)\colon 1\le \imath,\jmath\le d\right\}=\mat{d}.
\end{equation}
For $1\le\imath,\jmath\le d$ let $E_{\imath\jmath}\in\mat{d}$ be the standard matrix units. If $\gX_{\imath\jmath}^{(n)}$ for $1\le \imath,\jmath\le d$ are $n\times n$ generic matrices, then
$$\sum_{\imath,\jmath} E_{\imath\jmath}\otimes\gX_{\imath\jmath}^{(n)}$$
is a $(dn)\times (dn)$ generic matrix, so its determinant is irreducible for every $n\in\N$ \cite[Lemma B.2.10]{GW}. The polynomial
\begin{equation}\label{e:13}
\det\left(I-\sum_{\imath,\jmath} w_{\imath\jmath}(A)\otimes\gX_{\imath\jmath}^{(n)} \right)
\end{equation}
is a composition of the determinant of a generic matrix and an affine map on variables (this map is invertible by \eqref{e:120}), and therefore irreducible. Starting with polynomial \eqref{e:13} we $N=\sum_{\imath,\jmath}|w_{\imath\jmath}|-g$ times recursively apply Lemma \ref{l:main} to get rid of $w_{\imath\jmath}(A)$ for $|w_{\imath\jmath}|>1$ and conclude that $\det L\left(\gX^{(n)}\right)$ is an irreducible polynomial for every $n\ge (d^2-1) 2^N$.
\end{proof}

\begin{rem}
From the proof of Theorem \ref{t:main} one can derive a deterministic bound on $n$ for checking the irreducibility of $\det L\left(\gX^{(n)}\right)$ that is exponential in the size of $L$.
\end{rem}

\subsection{Irreducible free loci}\label{ss33}

Recall that a free locus is irreducible if it is not a union of smaller free loci. If $\fl$ is a free locus, than $\fl_n$ is either an empty set or a hypersurface for every $n\in\N$.

\begin{cor}\label{c:loc}
If $\fl$ is an irreducible free locus, then $\fl_n$ is an irreducible hypersurface for large $n$.
\end{cor}

\begin{proof}
Every irreducible free locus is a free locus of an irreducible pencil, so Theorem \ref{t:main} applies.
\end{proof}

\begin{exa}\label{ex1}
Let $A_1,A_2\in\mat{3}$ be as in Example \ref{ex0}. Note that
$$\det(I-\gx_1A_1-\gx_2A_2)=(1-\gx_1+\gx_2)(1-\gx_1-\gx_2).$$
Hence $L=I-A_1x_1-A_2x_2$ is an irreducible pencil and $\fl_1(L)$ is a union of two lines. On the other hand one can check that $\fl_2(L)$ is irreducible.
\end{exa}

Together with the block form \eqref{e:05} of a monic pencil $L$, Corollary \ref{c:loc} shows that components of $\fl_n(L)$ for large $n$ arise from the global decomposition of $\fl(L)$.

\begin{cor}\label{c:comp}
Let $L$ be a monic pencil and let $L_1,\dots, L_k$ be pairwise non-similar irreducible pencils appearing in \eqref{e:05}. For large $n$,
$$\fl_n(L)=\fl_n(L_1)\cup\cdots\cup \fl_n(L_k)$$
is the decomposition 
of $\fl_n(L)$
into distinct irreducible hypersurfaces.
\end{cor}

\begin{rem}\label{r:rad}
Let $L$ be an FL-minimal pencil. Let $L_1,\dots,L_\ell$ be pairwise non-similar irreducible pencils appearing in the decomposition of $L$ as in \eqref{e:05}. Since $\fl(L_i)\neq\fl(L_{i'})$ for $i\neq i'$, Theorem \ref{t:main} implies that $\det L_i(\gX^{(n)})$ are distinct irreducible polynomials for large $n$. Therefore
$$\det L(\gX^{(n)})=\prod_{i=1}^\ell \det L_i(\gX^{(n)})$$
is square-free for large $n$. Hence $\det L(\gX^{(n)})$ is a minimum degree defining polynomial for $\fl_n(L)$ and it generates the vanishing (radical) ideal of $\fl_n(L)$. 
\end{rem}

\section{Irreducibility of matrices over a free algebra}\label{s4}

In this section we extend Theorem \ref{t:main} to matrices over a free algebra. Let $\px$ be the free $\kk$-algebra generated by $\ulx=(x_1,\dots,x_g)$. Its elements are {\bf noncommutative polynomials}. For $f\in\opm_d(\px)$ let
$$\fl(f)=\bigcup_{n\in\N}\fl_n(f)\subseteq \cM^g, \qquad \text{where}\quad
\fl_n(f)=\left\{X\in\mat{n}^g\colon \det f(X)=0\right\},$$
be the {\bf free locus} of $f$. In Theorem \ref{t:irrpoly} below we prove that if $f$ does not factor in $\opm_d(\px)$, then $\fl_n(f)$ is an irreducible hypersurface for large $n$.

A fundamental finding of Cohn is that $\px$ is a free ideal ring, abbreviated fir, and that even rings with the weaker ``semifir'' property exhibit excellent behavior when it comes to factorizations. We list few definitions and facts about factorization of matrices over semifirs extracted from \cite[Chapter 3]{Coh}.

\begin{enumerate}
\item If $f\in \GL_d(\px)$, then $\det f(0)\neq0$; moreover, $\det f(\gX^{n})$ is a polynomial without zeros and hence constant, so $\det f(X)=\det f(0)^n$ for every $X\in\mat{n}^g$. In particular, irreducible monic pencils are non-invertible by \cite[Corollary 3.4]{KV}.

\item A matrix $f\in \opm_d(\px)$ is {\bf regular} if $f$ is not a zero divisor in $\opm_d(\px)$. In particular, if $f(0)=I$, then $f$ is regular. A regular non-invertible matrix is an {\bf atom} \cite[Section 3.2]{Coh} if it is not a product of two non-invertible matrices in $\opm_d(\px)$. In the special case $d=1$, a nonconstant noncommutative polynomial is an atom if it is not a product of two nonconstant noncommutative polynomials. We use this terminology to avoid  confusion with the notion of irreducibility for monic pencils.

\item We say that $f_1\in\opm_{d_1}(\px)$ and $f_2\in\opm_{d_2}(\px)$ are {\bf stably associated} if there exist $e_1,e_2\in\N$ with $d_1+e_1=d_2+e_2$ and $P,Q\in\GL_{d_1+e_1}(\px)$ such that
$$f_1\oplus I_{e_1} = P(f_2\oplus I_{e_2})Q.$$
Stable associativity is clearly an equivalence relation for regular (square) matrices over $\px$. By \cite[Corollary 0.5.5]{Coh}, $f_1$ and $f_2$ are stably associated if and only if
$${\px}^{d_1}/f_1 \cdot {\px}^{d_1}\cong {\px}^{d_2}/f_2 \cdot{\px}^{d_2}$$
as right $\px$-modules.

\item Let $f\in \opm_d(\px)$ be regular. By the definition of torsion modules \cite[Section 3.2]{Coh} and their relations to factorization \cite[Propositions 0.5.2 and 3.2.1]{Coh}, it follows that $f$ is an atom if and only if $\px^d/f \cdot\px^d$ has no non-trivial torsion submodules. In particular, if regular matrices $f_1$ and $f_2$ are stably associated, then $f_1$ is an atom if and only if $f_2$ is an atom.
\end{enumerate}

\begin{rem}\label{r:hom}
If $h_1,h_2\in\px$ are homogeneous and stably associated, then there exists $\lambda\in\kk^*=\kk\setminus\{0\}$ such that $h_2=\lambda h_1$. We now prove this. Let $d\in\N$ be such that
\begin{equation}\label{e:h1}
P(h_1\oplus I)=(h_2\oplus I)Q
\end{equation}
for some $P,Q\in\GL_d(\px)$. Then it is easy to see that $h_1$ and $h_2$ are of the same degree $\delta$. Let $P_i$ and $Q_i$ be homogeneous parts of $P$ and $Q$, respectively, of degree $i$. By looking at the constant part of \eqref{e:h1} we obtain $P_0(0\oplus I)=(0\oplus I)Q_0$, so the first row of $P_0$ equals $(\alpha\, 0\dots\, 0)$ for some $\alpha\in\kk$ and the first column of $Q_0$ equals $(\beta\, 0\dots\, 0)^\ti$ for some $\beta\in\kk$. Moreover, $\alpha,\beta\in\kk^*$ since $P_0,Q_0\in\GL_d(\kk)$. Next, the homogeneous part of \eqref{e:h1} of degree $\delta$ equals
$$P_0(h_1\oplus 0)+P_\delta (0\oplus I)=(h_2\oplus 0)Q_0+(0\oplus I) Q_\delta.$$
Note that the first column of $P_\delta (0\oplus I)$ and the first row of $(0\oplus I) Q_\delta$ are zero. Since the $(1,1)$-entries of $P_0(h_1\oplus 0)$ and $(h_2\oplus 0)Q_0$ are $\alpha h_1$ and $\beta h_2$, respectively, we can take $\lambda=\alpha\beta^{-1}$.
\end{rem}

\begin{lem}\label{l:sa}
Let $f\in\opm_d(\px)$ and $f(0)=I$. If $f$ is an atom, then $f$ is stably associated to an irreducible monic pencil $L$.
\end{lem}

\begin{proof}
By linearization, also known as Higman's trick \cite[Section 8.5]{Coh}, we have
\begin{equation}\label{e:higman}
\begin{pmatrix}1 & a_1\\ 0 & 1\end{pmatrix}
\begin{pmatrix}a_0+a_1a_2 & 0 \\ 0 & 1\end{pmatrix}
\begin{pmatrix}1 & 0\\ -a_2 & 1\end{pmatrix}
=\begin{pmatrix}a_0 & a_1 \\ -a_2 & 1\end{pmatrix}
\end{equation}
for all square matrices $a_0,a_1,a_2$ of equal sizes. Using \eqref{e:higman} we can step-by-step ``decouple'' products appearing in $f$ to obtain
\begin{equation}\label{e:42}
P\left(f\oplus I\right)Q = L
\end{equation}
for some monic linear pencil $L$ of size $d'$ and $P,Q\in\GL_{d'}(\px)$. We remark that $P$ (resp. $Q$) is upper (resp. lower) unitriangular. Hence $f$ and $L$ are stably associated, so $L$ is an atom in $\opm_{d'}(\px)$. As in \eqref{e:05}, there exists $U\in\GL_{d'}(\kk)$ such that
\begin{equation}\label{e:43}
ULU^{-1}=
\begin{pmatrix}
L_1 & \star& \cdots & \star\\
& \ddots & \ddots & \vdots\\
& & \ddots & \star\\
& & & L_{\ell}
\end{pmatrix},
\end{equation}
where each $L_k$ is either $I$ or an irreducible monic pencil. Since $f$ is not invertible, at least one of $L_k$ is irreducible (i.e., $L_k\neq I$); let $\ell_0$ be the largest such $k$. By multiplying $ULU^{-1}$ on the left-hand side with an appropriate invertible matrix (note that the block, which is below and to the right of $L_{\ell_0}$, is invertible) we see that $L$ (and thus $f$) is stably associated to
$$L'=
\begin{pmatrix}
L_1 & \cdots & \star\\
& \ddots & \vdots\\
& & L_{\ell_0}
\end{pmatrix}.
$$
If $L_k\neq I$ for some $k<\ell_0$, then
$$L'=
\begin{pmatrix}
I & & \\
& \ddots & \\
& & L_{\ell_0}
\end{pmatrix}
\begin{pmatrix}
L_1 & \cdots & \star\\
& \ddots & \vdots\\
& & I
\end{pmatrix}
$$
would be a product of non-invertible matrices, contradicting that $L'$ is an atom. Therefore
$$L'=
\begin{pmatrix}
I & \cdots & \star\\
& \ddots & \vdots\\
& & L_{\ell_0}
\end{pmatrix}
=
\begin{pmatrix}
I & & \\
& \ddots & \\
& & L_{\ell_0}
\end{pmatrix}
\begin{pmatrix}
I & \cdots & \star\\
& \ddots & \vdots\\
& & I
\end{pmatrix},
$$
so $L'$ (and thus $f$) is stably associated to $L_{\ell_0}$.
\end{proof}

Let $f\in\opm_d(\px)$ be regular. Then $f$ admits a factorization $f=f_1\cdots f_\ell$ into atoms $f_i\in\opm_d(\px)$ that are unique up to stable associativity by \cite[Proposition 3.2.9]{Coh}, and thus
$$\fl(f)=\fl(f_1)\cup\cdots\cup\fl(f_\ell).$$
If we change the order of atomic factors in a factorization of $f$, we possibly obtain a different noncommutative polynomial with the same free locus as $f$. Also, if an atomic factor is replaced by a (power of a) stably associated element in $\opm_d(\px)$, the polynomial can change but its free locus still equals $\fl(f)$.

We are now ready for our main result. Recall that $\gX^{(n)}$ is a $g$-tuple of $n\times n$ generic matrices.

\begin{thm}[Polynomial Singularit\"atstellensatz]\label{t:irrpoly}
For $i\in\{1,2\}$ let $f_i\in \opm_{d_i}(\px)$ satisfy $f_i(0)=I$.
\begin{enumerate}
	\item If $f_1$ is an atom, then $\det f_1(\gX^{(n)})$ is an irreducible polynomial for large $n$.
	\item If $f_1$ and $f_2$ are atoms and $\fl(f_1)=\fl(f_2)$, then $f_1$ and $f_2$ are stably associated.
	\item $\fl(f_1)\subseteq\fl(f_2)$ if and only if 
each  atomic factor of $f_1$ is up to stable associativity an atomic factor of $f_2$.
\end{enumerate}
\end{thm}

\begin{proof}
(1) This is a direct consequence of Lemma \ref{l:sa} and Theorem \ref{t:main}.

(2) By Lemma \ref{l:sa}, $f_i$ is stably associated to an irreducible pencil $L_i$ for $i=1,2$. Since $\fl(L_1)=\fl(L_2)$, $L_1$ and $L_2$ are similar, so $f_1$ and $f_2$ are stably associated.

(3) Follows by (2) and a factorization of $f_i$ into atoms.
\end{proof}

\begin{rem}\label{r:shift}
The conclusions of Theorem \ref{t:irrpoly} hold more generally for $f\in\opm_d(\px)$ satisfying $f(\kk^g)\cap \GL_{d}(\kk)\neq\emptyset$, which readily follows from translating $\ulx$ by a scalar point and multiplying $f$ with an invertible matrix.
\end{rem}

For scalar noncommutative polynomials and linear matrix pencils we also give an effective converse to Theorem \ref{t:irrpoly}(1).

\begin{prop}\label{p:conv}
For $\delta>1$ and $d>1$ set
$$n_1=\left\lceil \frac{\delta}{2}\right\rceil,\qquad
n_2=
\left\{\begin{array}{ll}
1 & d=2, \\
\left\lceil (d-1)\sqrt{\frac{2(d-1)^2}{d-2}+\frac{1}{4}}+\frac{d-1}{2}-2 \right\rceil& d\ge3.
\end{array}\right.$$
\begin{enumerate}
	\item Let $f\in\px$ be of degree $\delta$ and $f(0)=1$. If $\det f(\gX^{(n)})$ is irreducible for some $n\ge n_1$, then $f$ is an atom.
	\item Let $L$ be a monic pencil of size $d$. If $\det L(\gX^{(n)})$ is irreducible for some $n\ge n_2$, then $L$ is an atom.
\end{enumerate}
\end{prop}

\begin{proof}
(1) If $f=f_1f_2$ for non-constant $f_1,f_2\in\px$, then $\deg f_i\le \delta-1$, so it suffices to show that $\det f_i(\gX^{(n)})$ is not constant for $i=1,2$. Suppose $\det f_i(\gX^{(n)})=1$. By the Cayley-Hamilton theorem we have
\begin{equation}\label{e:ch}
f_1(\gX^{(n)})\adj f_1(\gX^{(n)})=1,
\end{equation}
where $\adj M$ denotes the adjugate of a square matrix $M$ \cite[Section 1.9]{MN}. Note that $f_1(\gX^{(n)})$ and $\adj f_1(\gX^{(n)})$ both belong to the {\it trace ring} of $n\times n$ matrices, a unital $\kk$-algebra $R$ generated by $\gX^{(n)}_j$ and $\tr(w(\gX^{(n)}))$ for $w\in\mx$. By \cite[Section 5]{For}, $R$ is a graded domain, where the grading is imposed by the total degree in $\opm_{n}(\kk[\ulxi])$. Therefore \eqref{e:ch} implies $\deg f_1(\gX^{(n)})=0$ and so $f_1(\gX^{(n)})=1$. Hence $1-f_1$ is a polynomial identity for $n\times n$ matrices of degree at most $\delta-1$, which contradicts $2n>\delta-1$. 

(2) If $L$ is not an atom, then $\det L(\gX^{(n)})=\det L_1(\gX^{(n)})\det L_2(\gX^{(n)})$ for monic pencils $L_i$ of size at most $d-1$ whose coefficients are not jointly nilpotent. Then $\det L_i(\gX^{(n)})$ is not constant for $i=1,2$ by \cite[Proposition 3.3]{KV}.
\end{proof}

\section{Pencils with polynomial free locus and matrix perturbations}\label{s6}

The purpose of this section is twofold. In Subsection \ref{ss63} we characterize monic pencils whose free locus is the free locus of a noncommutative polynomial using state space realization theory. This leads to an efficient algorithm for checking the equality of free loci presented in Subsection \ref{ss64}. Using properties of minimal factorizations of realizations we then obtain an intriguing statement about invariant subspaces of rank-one perturbations of jointly nilpotent matrices (Theorem \ref{t:pert}).

\subsection{Perturbations and invariant subspaces}\label{ss61}
 
 While perturbations of matrices is a classical
theory \cite{Kat},  our path takes a different direction than the classical ones. Given a $g$-tuple $A=(A_1,\dots,A_g)$ of matrices in $\mat{d}$ we shall consider  perturbations
called {\bf right rank-one perturbations}, namely, ones of the form
$$  A_j + b_j c^\ti  \qquad j= 1, \dots, g$$
for $b_j,c \in \kk^d$. The perturbation is called {\bf non-degenerate} if
\begin{enumerate}[(i)]
	\item $\{b_1,\dots,b_g\}$ is not contained in a non-trivial invariant subspace for $A$,
	\item $c$ does not lie in a non-trivial invariant subspace for $A^\ti$.
\end{enumerate}

One  consequence of  Theorem \ref{t:irrpoly} relates invariant subspaces for 
a matrix tuple to invariant subspaces for its right rank-one perturbations.

\begin{thm}\label{t:pert}
Let $A\in\mat{d}^g$ be a tuple of jointly nilpotent matrices. If $\cS$ is an invariant subspace of a non-degenerate right rank-one
perturbation of $A$, then there is a complementary space
$\cS^\times$, $\kk^d=\cS\dotplus \cS^\times$,
 which is invariant under $A$.
\end{thm}

To prove this we shall use state space systems realizations, a technique closely related to the ``linearizations"
introduced in the proof of Lemma \ref{l:sa}. The proof of Theorem \ref{t:pert} concludes in Subsection \ref{sss:pfPert}.

\subsection{State space realizations}\label{ss62}
Next we introduce some essential background.
Let $\rx$ be the free skew field \cite{Coh} and $\rx_0\subset\rx$ the subring of noncommutative rational functions that are regular at the origin \cite{KVV1,Vol}.  Each $\rr\in\rx_0$ admits a {\bf noncommutative Fornasini–Marchesini state space realization} (shortly an FM-realization)
\begin{equation}\label{e:fm}
\rr=\delta+c^\ti L^{-1}b,
\end{equation}
where $\delta\in\kk$, $c\in\kk^d$, $b=\sum_jb_jx_j$ for $b_j\in\kk^d$ and $L=I-\sum_jA_jx_j$ for $A_j\in\mat{d}$; see e.g. \cite[Section 2.1]{BGM}. Here $d\in\N$ is the size of realization \eqref{e:fm}. 

For future use we recall some well-known facts about minimal FM-realizations.

\begin{enumerate}
\item We say that the realization \eqref{e:fm} is {\bf controllable} if
$$\spa\left\{w(A)b_j\colon w\in\mx, 1\le j\le g\right\}=\kk^d$$
and {\bf observable} if
$$\spa\left\{w(A)^\ti c\colon w\in\mx\right\}=\kk^d.$$
By \cite[Theorem 9.1]{BGM}, the realization \eqref{e:fm} is minimal if and only if it is observable and controllable.
\item A minimal realization of $\rr$ is unique up to similarity \cite[Theorem 8.2]{BGM}.

\item The domain of $\rr$ is precisely the complement of $\fl(L)$ if \eqref{e:fm} is minimal by \cite[Theorem 3.1]{KVV1} and \cite[Theorem 3.10]{Vol1}.
\item
Lastly, if $\rr(0)=\delta\neq0$ (equivalently, $\rr^{-1}\in\rx_0$), then by \cite[Theorem 4.3]{BGM} we have an FM-realization
\begin{equation}\label{e:fmi}
\rr^{-1}=\delta^{-1}+(-\delta^{-1}c^\ti)\left(L^\times\right)^{-1}(\delta^{-1}b),
\end{equation}
of $\rr^{-1}$, where $L^\times=I-\sum_jA^\times_jx_j$ and
\begin{equation}\label{e:cross}
A^\times_j=A_j-\delta^{-1}b_jc^\ti.
\end{equation}
Because the realizations \eqref{e:fm} and \eqref{e:fmi} are of the same size, we see that \eqref{e:fm} is minimal for $\rr$ if and only if \eqref{e:fmi} is minimal for $\rr^{-1}$.
\end{enumerate}

\begin{rem}
The linearization trick \eqref{e:42} when inverted gives a
type of a realization: if $e^\ti= (1, 0, \dots, 0)$, then
$$f^{-1} = e^\ti \left(f^{-1} \oplus I \right) e =
e^\ti Q^{-1}  L^{-1} P^{-1} e=
e^\ti L^{-1} e
$$
using that $P^{-1} $ (resp. $Q^{-1} $) is upper (resp. lower) unitriangular. Recall that $L$ is a monic linear pencil. This is often called a descriptor realization of $f^{-1}$.

To convert to an FM-realization we use the assumption $f(0) =1$ 
which makes $\delta =1$. Then
$$e^\ti L^{-1} e -1= e^\ti L^{-1} (I - L)e,$$
so $f^{-1}$ has an FM-realization with $L=I-\sum_jA_jx_j$, $c=e$, $b_j=A_je$ and $\delta =1$. Most importantly, the representing pencils for the FM and descriptor realizations are the same.
\end{rem}

\subsection{Flip-poly pencils}\label{ss63}
Next we define the pencils to which we shall associate polynomial free loci. A monic pencil $L=I-\sum_jA_jx_j$ is {\bf flip-poly} if $A_j=N_j+E_j$, where $N_j$ are jointly nilpotent matrices and $\codim(\bigcap_j \ker E_j)\le1$.

\begin{lem}\label{l:flip1}
Let $f\in\px$ and $f(0)=1$. If $L$ is a monic pencil appearing in a minimal realization of $f^{-1}$, then $L$ is flip-poly, the intersection of kernels of its coefficients is trivial and $\det f(\gX^{(n)})=\det L(\gX^{(n)})$ for all $n\in\N$.
\end{lem}

\begin{proof}
Let $f=1-c^\ti L_0^{-1}b$ with $L_0=I-\sum_jN_jx_j$ be a minimal realization. Since $\dom f=\cM^g$, $L_0$ is invertible at every matrix point, so $\fl(L_0)=\emptyset$ and hence $N_j$ are jointly nilpotent matrices by \cite[Corollary 3.4]{KV}. Since minimal realizations are unique up to similarity, we can assume that $L=I-\sum_jA_jx_j$ where $A_j=N_j+b_jc^\ti$. Then $L$  is flip-poly and $f^{-1}=1+c^\ti L^{-1}b$ is a minimal realization by \eqref{e:fmi}.

Since
$$
\begin{pmatrix}
I & 0 \\ c^\ti L_0^{-1} & 1
\end{pmatrix}
\begin{pmatrix}
L_0 & 0 \\ 0 & f
\end{pmatrix}
\begin{pmatrix}
I & L_0^{-1}b \\ 0 & 1
\end{pmatrix}
=
\begin{pmatrix}
L_0 & b \\ c^\ti & 1
\end{pmatrix}
=
\begin{pmatrix}
I & b \\ 0 & 1
\end{pmatrix}
\begin{pmatrix}
L & 0 \\ 0 & 1
\end{pmatrix}
\begin{pmatrix}
I & 0 \\ c^\ti & 1
\end{pmatrix}
$$
and $N_j$ are jointly nilpotent, we have
$$\det f(\gX^{n})=\det L_0(\gX^{(n)}) \det f(\gX^{(n)})=\det L(\gX^{(n)})$$
for all $n\in\N$.

Lastly suppose there exists a nonzero $v\in\bigcap_j\ker A_j$. Since the realization $1+c^\ti L^{-1}b$ is observable, it follows that $c^\ti v\neq0$. Because $N_jv+(c^\ti v)b_j=A_jv=0$ holds for all $j$, we conclude that
$$\spa\left\{w(N)b_j\colon w\in\mx, 1\le j\le g\right\}\subseteq \sum_j \ran N_j,$$
which contradicts controllability of $1-c^\ti L_0^{-1}b$. Hence $\bigcap_j\ker A_j=\{0\}$.
\end{proof}

\begin{prop}\label{p:flip2}
For every flip-poly pencil $L$ there exists a flip-poly pencil $L_0$ such that $\fl(L)=\fl(L_0)$ and $L_0$ appears in a minimal realization of $f^{-1}$ for some $f\in\px$ with $f(0)=1$.
\end{prop}

\begin{proof}
By assumption we have $L=I-\sum_j(N_j-b_jc^\ti)x_j$ for some $b_j,c\in\kk^d$ and jointly nilpotent $N_j\in\mat{d}$. Let
\begin{equation}\label{e:f0}
f=1+c^\ti\left(I-\sum_jN_jx_j\right)^{-1}\left(\sum_jb_jx_j\right).
\end{equation}
Since $N_j$ are jointly nilpotent, $f$ is a noncommutative polynomial. By \eqref{e:fmi} we have
\begin{equation}\label{e:min}
f^{-1}=1-c^\ti L^{-1}\left(\sum_jb_jx_j\right).
\end{equation}
While the realization \eqref{e:min} is not necessarily minimal, it suffices to prove that its minimization results in a realization with a flip-poly pencil $L_0$ satisfying $\fl(L_0)=\fl(L)$.

Recall that the minimization algorithm comprises of two steps. A starting realization $c^\ti(I-\sum_jA_jx_J)^{-1}b$ is first restricted to the invariant subspace $\spa\{w(A)b_j\}_{w,j}$. The resulting realization $c'^\ti(I-\sum_jA_j'x_J)^{-1}b'$ is then restricted to the invariant subspace $\spa\{w(A')^\ti c'\}_w$, which yields a minimal realization.

Denote $A_j=N_j-b_jc^\ti$ and assume $\cS_1=\spa\{w(A)b_j\}_{w,j}\neq\kk^d$. Let $\cS_2\subset\kk^d$ be a complementary space of $\cS_1$, i.e., $\cS_1\dotplus \cS_2=\kk^d$. For $i=1,2$ let $\iota_i:\cS_i\to\kk^d$ and $\pi_i:\kk^d\to \cS_i$ be the corresponding embeddings and projections, respectively. Then $\pi_1\circ L \circ \iota_1$ is the pencil produced after the first step of minimization and
$$\fl(L)=\fl(\pi_1\circ L \circ \iota_1)\cup\fl(\pi_2\circ L \circ \iota_2).$$
Since $A_jv=N_jv-(c^\ti v)b_j$ for every $v\in\kk^d$ and $b_j\in\cS_1$, we see that $\cS_1$ is also invariant under $N_j$. Therefore $\pi_1\circ L \circ \iota_1$ is a flip-poly pencil. On the other hand we have $A_jv-N_jv\in\cS_1$ for every $v\in\kk^d$, so $\pi_2\circ A_j\circ \iota_2$ are jointly nilpotent and hence $\fl(\pi_2\circ L \circ \iota_2)=\emptyset$. 

Analogous reasoning holds for the second step of minimization, so the pencil appearing in a minimal realization of $f^{-1}$ is flip-poly and its free locus equals $\fl(L)$.
\end{proof}

\begin{cor}\label{c:polyloc}
The set of free loci of noncommutative polynomials coincides with the set of free loci of flip-poly pencils.
\end{cor}

\begin{proof}
Direct consequence of Lemma \ref{l:flip1} and Proposition \ref{p:flip2}.
\end{proof}

\begin{prop}\label{p:nonflip}
Let $L$ be an irreducible pencil. If $L$ is not flip-poly, then $\fl(L)\neq\fl(f)$ for all $f\in\px$.
\end{prop}

\begin{proof}
Suppose $\fl(L) = \fl(f)$ for some $f\in\px$. Since $L$ is irreducible, we can assume that $f(0)=1$ and $f$ is irreducible by Theorems \ref{t:main} and \ref{t:irrpoly}. Let $L'$ be a monic pencil appearing in the realization of $f^{-1}$. By Lemma \ref{l:flip1} we have
$$L (\gX^{(n)})=\det f(\gX^{(n)})=\det L' (\gX^{(n)})$$
for all $n\in\N$. Since the intersection of kernels of coefficients of $L'$ is trivial, we deduce that $L'$ is irreducible by writing it in the form \eqref{e:05}. Therefore $L$ and $L'$ are similar and hence $L$ is flip-poly.
\end{proof}

\begin{exa}
Assume $A\in\mat{d}$ has an eigenvalue $\lambda\neq0$ with geometric multiplicity at least 2 and let $b,c\in\kk^d$ be arbitrary. Then an easy calculation shows that $\lambda$ is also an eigenvalue of $A+b c^\ti$, so $A+b c^\ti$ is not nilpotent. Therefore a monic pencil having $A$ as one of its coefficients is not flip-poly.

In particular, if $d\ge 3$, then there exist $A_1,A_2\in\mat{d}$ such that $A_1$ has a nonzero eigenvalue with geometric multiplicity at least 2 and $A_1,A_2$ generate $\mat{d}$. For example, one can choose $A_2$ to be the permutation matrix corresponding to the cycle $(1\ 2\ \dots\ d)$ and $A_1=\diag(1,\dots,1,-1)$. Then $L=I-A_1x_1-A_2x_2$ is an irreducible pencil that is not flip-poly, so $\fl(L)\neq \fl(f)$ for all $f\in\px$ by Proposition \ref{p:nonflip}.
\end{exa}

\begin{exa}
Let $L=I-A_1 x_1-A_2 x_2$ for $A,A_2\in \mat{2}$; we claim $\fl(L)=\fl(f)$ for some $f\in\px$ of degree at most $2$.

Looking at the zeros of the polynomial $\det(\gx_1 A_1+\gx_2 A_2)\in\kk[\gx_1,\gx_2]$ we see that there exists a nonzero $u\in\kk^2$ such that $A_1u$ and $A_2u$ are linearly dependent. If $A_1u,A_2u\in\kk \cdot u$, then $A_1$ and $A_2$ have a common eigenvector, so clearly $\fl(L)=\fl(\ell_1\ell_2)$ for some affine linear $\ell_i\in\px$. Otherwise we have $A_1u,A_2u\in\kk \cdot v$ for some $v\in\kk^2\setminus\kk\cdot u$. With respect to the basis $\{u,v\}$ of $\kk^2$ we have
$$A_j=\begin{pmatrix}
0 & \alpha_{j1} \\ \alpha_{j2} & \alpha_{j3}
\end{pmatrix}=
\begin{pmatrix}
0 & 0 \\ \alpha_{j2} & 0
\end{pmatrix}+
\begin{pmatrix} \alpha_{j1} \\ \alpha_{j3} \end{pmatrix}
\begin{pmatrix}0 & 1\end{pmatrix},$$
so $L$ is flip-poly and hence $L=\fl(f)$ for some $f\in\px$ of degree at most $2$ by \eqref{e:fmi}.
\end{exa}

\subsection{Minimal factorizations}\label{ss65}

A factorization $\rr=\rr_1\rr_2$ for $\rr,\rr_1,\rr_2\in\rx_0$ is {\bf minimal} if the size of the minimal realization of $\rr$ equals the sum of the sizes of minimal realizations of $\rr_1$ and $\rr_2$. That is, if $\rr_i=1+c_i^\ti L_i^{-1}b_i$ is a minimal realization for $i=1,2$, then $\rr=\rr_1\rr_2$ is a minimal factorization if and only if
\begin{equation}\label{e:minfac}
\rr=1+\begin{pmatrix}c_1^\ti & c_2^\ti \end{pmatrix}
\begin{pmatrix}
L_1 & b_1 c_2^\ti \\ 0 & L_2
\end{pmatrix}^{-1}
\begin{pmatrix}b_1 \\ b_2 \end{pmatrix}
\end{equation}
is a minimal realization by \cite[Theorem 4.1]{BGM}.

Let $\rr=\delta+c^\ti L^{-1}b$ be a minimal realization of size $d$ and $\delta\neq0$. In \cite[Section 4]{KVV1} it is explained that by the multivariable noncommutative version of \cite[Theorem 9.3]{BGKR}, minimal factorizations of $\rr$ are in one-to-one correspondence with pairs $(\cS,\cS^\times)$ of subspaces in $\kk^d$ such that
\begin{enumerate}[(a)]
	\item $\cS$ is invariant under $A_1,\dots,A_g$,
	\item $\cS^\times$ is invariant under $A_1^\times,\dots,A_g^\times$,
	\item $\cS\dotplus\cS^\times =\kk^d$.
\end{enumerate}

\begin{prop}\label{p:polyfac}
Let $f\in\px$ and $f(0)\neq0$. 
Minimal factorizations of $f$ are precisely polynomial factorizations of $f$.
\end{prop}

\begin{proof}
Let $f=\rr'\rr''$ be a minimal factorization. Then $\dom f=\dom \rr'\cap\dom\rr''$ by \cite[Theorem 4.2]{KVV1}. Consequently $\dom f=\cM^g$ implies $\dom\rr'=\dom\rr''=\cM^g$, so $\rr',\rr''\in\px$ by \cite[Theorem 4.2]{KV}.

Let $f=f'f''$ be a polynomial factorization; 
without loss of generality let $f(0)=f'(0)=f''(0)=1$. 
As already mentioned in the proof of Lemma \ref{l:flip1}, 
the coefficients of the pencil appearing in a minimal realization of a noncommutative polynomial are jointly nilpotent. 
By \eqref{e:minfac} it thus suffices to prove the following: if
\begin{equation}\label{e:51}
1+c'^\ti\left(I-\sum_j N'_jx_j\right)^{-1}\left(\sum_jb'_jx_j\right),\qquad
1+c''^\ti\left(I-\sum_j N''_jx_j\right)^{-1}\left(\sum_jb''_jx_j\right)
\end{equation}
are minimal realizations of size $d'$ and $d''$, respectively, 
where $N_j'$ are jointly nilpotent and $N''_j$ are jointly nilpotent, 
then the ``product realization''
\begin{equation}\label{e:52}
1+c^\ti\left(I-\sum_j N_jx_j\right)^{-1}\left(\sum_jb_jx_j\right),
\end{equation}
where
\begin{equation}\label{e:53}
c=\begin{pmatrix}c' \\ c''\end{pmatrix},\qquad
N_j=\begin{pmatrix}N'_j & b'_jc''^\ti \\ 0 & N''_j\end{pmatrix},\qquad
b_j=\begin{pmatrix}b'_j \\ b''_j\end{pmatrix},
\end{equation}
is minimal.

We now prove that the product system \eqref{e:52} is observable. Let $w_0\in\mx$ and $1\le j_0\le g$ be such that $\alpha=c'^\ti w_0(N')b'_{j_0}\neq0$ and $c'^\ti w(N')b'_j=0$ for all $|w|>|w_0|$ and $1\le j\le g$. Since $\spa\{w(N')b'_j\}_{w,j}=\kk^{d'}$, it follows that $c'^\ti w(N')=0$ for all $|w|>|w_0|$. Denote $w_1=w_0x_{j_0}$. We claim that
\begin{equation}\label{e:54}
\spa\left\{\alpha w(N'')^\ti c''+\sum_{i=0}^{|w_1|-1}\beta_i(w_1^{i:}w)(N'')^\ti c''\colon w\in\mx\right\}
=\kk^{d''}
\end{equation}
for arbitrary choice of $\beta_1,\dots,\beta_{|w_1|}\in\kk$. Here $w_1^{i:}$ is obtained by removing the first $i$ letters in $w$. Indeed, by induction on $k=1,\dots,d''$ we show that the sets of rows
$$\left\{\alpha c''^\ti w(N'')+\sum_{i=0}^{|w_1|-1}\beta_ic''^\ti (w_1^{i:}w)(N'')
\colon |w|\ge d''-k\right\},\qquad \left\{c''^\ti w(N'')\colon |w|\ge d''-k\right\}$$
span the same subspace and then \eqref{e:54} follows by $\spa\{w(N'')^\ti c''\colon w\in\mx\}=\kk^{d''}$, which holds by the minimality assumption.

A routine computation next shows that for $w=x_{j_1}\cdots x_{j_\ell}$ we have
\begin{equation}\label{e:55}
w(N)=\begin{pmatrix}
w(N') & \sum_{w=ux_jv} u(N')b'_{j}c''^\ti v(N'') \\
0 & w(N'')
\end{pmatrix}.
\end{equation}
Hence
\begin{equation}\label{e:57}
c^\ti w(N)=\begin{pmatrix} c'^\ti w(N') & \star\end{pmatrix}
\end{equation}
holds for all $w\in\mx$. On the other hand, \eqref{e:55} also implies
\begin{equation}\label{e:56}
c^\ti(w_1w)(N)=\begin{pmatrix} 0 & \alpha c''^\ti w(N'')
+\sum_{i=0}^{|w_1|-1}\beta_ic''^\ti (w_1^{i:}w)(N'')\end{pmatrix}
\end{equation}
for all $w\in\mx\setminus\{1\}$, where $\beta_i\in\kk$ depend on $w$. Finally, since the first realization in \eqref{e:51} is observable and \eqref{e:54} holds, \eqref{e:57} and \eqref{e:56} imply
$$\spa\left\{w(N)^\ti c\colon w\in\mx \right\}=\kk^{d'+d''},$$
hence \eqref{e:51} is an observable realization. By an analogous argument we check controllability, so \eqref{e:51} is a minimal FM-realization.
\end{proof}

\subsubsection{Proof of Theorem \ref{t:pert}}\label{sss:pfPert}

Let $B_j=A_j+b_jc^\ti$ for $j=1,\dots,g$ and
$$L_0=I-\sum_jA_jx_j,\qquad L=I-\sum_jB_jx_j,\qquad b=\sum_jb_jx_j.$$
In the language of FM-realizations, $A_j$ being nilpotent and $(B_j)_j$ being a non-degenerate perturbation of $(A_j)_j$ means that $1-c^\ti L_0^{-1}b$ is a minimal realization of a nonconstant $f\in\px$. By \eqref{e:fmi}, $1+c^\ti L^{-1}b$ is a minimal realization of $f^{-1}$.

If $\cS$ is a non-trivial invariant subspace for the $B_j$, then $L$ is similar to
$$\begin{pmatrix}L' & \star \\ 0 & L''\end{pmatrix}$$
for monic pencils $L'$ and $L''$. We claim that the coefficients of $L'$ and $L''$ are not jointly nilpotent. For this to hold we need to show that $B_j|_{\cS}$ are not jointly nilpotent and that the induced operators $\tilde{B}_j:\kk^d/\cS\to\kk^d/\cS$ are not jointly nilpotent.

If $B_j|_{\cS}$ are jointly nilpotent, there exists $v\in\cap_j\ker B_j\setminus\{0\}$. Since $\{w(B)^\ti c\}_w=\kk^d$, we have $c^\ti v\neq0$. Then $A_jv=B_jv-(c^\ti v)b_j$ implies $b_j\in\sum_j\ran A_j$, so
$$\left\{w(A)b_j\right\}_{w,j}\subseteq\sum_j\ran A_j\neq\kk^d$$
because $A_j$ are jointly nilpotent, which contradicts minimality.

If $\tilde{B_j}$ are jointly nilpotent, then $\sum_j\ran\tilde{B}_j\neq\kk^d/\cS$ and hence $\sum_j\ran B_j\neq\kk^d$. Since $A_j$ are jointly nilpotent, there exists $v\in\bigcap_j\ker A_j\setminus\{0\}$ and $c^\ti v\neq0$ because $\{w(A)^\ti c\}_w=\kk^d$. Therefore $B_jv=(c^\ti v)b_j$ implies
$$\left\{w(B)b_j\right\}_{w,j}\subseteq\sum_j\ran B_j\neq\kk^d,$$
which contradicts minimality.

By Lemma \ref{l:flip1} we have
$$\det f(\gX^{n})= \det L(\gX^{(n)})=\det L'(\gX^{(n)})\det L''(\gX^{(n)})$$
for all $n\in\N$. This factorization is non-trivial for large $n$ because the coefficients of $L'$ and $L''$ are not jointly nilpotent. Therefore $f$ is not an atom in $\px$ by Theorem \ref{t:irrpoly}(1). 

For a moment assume that $\cS$ is an irreducible invariant subspace for $B_j$ (that is, $\cS$ does not contain a smaller nonzero invariant subspace). If $f=f_\ell\cdots f_1$ is a factorization of $f$ into atoms, then it is a minimal factorization by Proposition \ref{p:polyfac}, so $f^{-1}=f_1^{-1}\cdots f_\ell^{-1}$ is also a minimal factorization. If $f_i^{-1}=1+c_i L_i^{-1}b_i$ is a minimal realization, then
\begin{equation}\label{e:big}
1+\begin{pmatrix}c_1 & c_2 & \cdots & c_\ell\end{pmatrix}
\begin{pmatrix}
L_1 & b_1c_2^\ti & \cdots & b_1c_\ell^\ti \\
& L_2 & \cdots & b_2c_\ell^\ti \\
& & \ddots & \vdots \\
& & & L_\ell
\end{pmatrix}^{-1}
\begin{pmatrix}b_1 \\ b_2 \\ \vdots \\ b_\ell\end{pmatrix}
\end{equation}
is a minimal realization of $f^{-1}$ by \eqref{e:minfac}. The block structure of \eqref{e:big} gives us a chain of invariant subspaces for $B_j$
$$\{0\}=\cV_0\subsetneq\cV_1\subsetneq\cdots\subsetneq\cV_\ell=\kk^d$$
such that $\cV_{i+1}/\cV_i$ are irreducible (for the action of linear maps on quotient spaces induced by $B_j$). We claim that after a basis change preserving the structure of \eqref{e:big} we can assume that $\cS=\cV_1$. Indeed, if $\cV_{i-1}\subsetneq \cS\subseteq \cV_i$, then $\cV_i=\cV_{i-1}\dotplus \cS$ and therefore $b_{i'}c_i^\ti=0$ for all $i'< i$; applying the basis change corresponding to switching $\cV_{i-1}$ and $\cS$ thus preserves the structure of \eqref{e:big} and results in replacing $\cV_1$ by $\cS$. Hence $\cS$ determines a minimal factorization $f^{-1}=f_1^{-1}h^{-1}$, where $f_1$ is an atom and $h$ is a nonconstant polynomial (since $f$ is not an atom). Due to the correspondence between minimal factorizations and pairs of invariant subspaces it follows that there exists an invariant subspace $\cS^\times$ for $B_j^\times=B_j-b_jc^\ti=A_j$ that is complementary to $\cS$.

Now let $\cS$ be an arbitrary non-trivial invariant subspace for $B_j$. Then we can find a chain of invariant subspaces
$$0=\cS_0\subsetneq \cS_1\subsetneq\cdots\subsetneq\cS_m=\cS$$
such that $\cS_{i+1}/\cS_i$ are irreducible. Inductively applying the reasoning from the previous paragraph we see that the sequence of (quotient) spaces $\cS_1,\cS_2/\cS_1,\dots,\cS_m/\cS_{m-1}$ yields a minimal factorization $f^{-1}=f_1^{-1}\cdots f_m^{-1}h^{-1}$. Hence $\cS$ determines a minimal factorization $f^{-1}=(f_m\cdots f_1)^{-1}h^{-1}$ and we obtain $\cS^\times$ as in the previous paragraph.
\qed

\begin{rem}
Observe that Theorem \ref{t:irrpoly}(1) (at least for scalar noncommutative polynomials) can be deduced from Theorem \ref{t:pert} without using Cohn's semifir factorization theory. Indeed, let $f\in\px$ satisfy $f(0)=1$ and let $f^{-1}=1+c^\ti L^{-1}b$ be a minimal FM-realization. By Lemma \ref{l:flip1} we have $\det f(\gX^{(n)})=\det L(\gX^{(n)})$, so if $\det f(\gX^{(n)})$ is not irreducible for large $n$, the monic pencil $L$ is not irreducible by Theorem \ref{t:irrpoly}. Therefore the coefficients of $L$ have a non-trivial invariant subspace $\cS$, so the assumptions of Theorem \ref{t:pert} are met and from it we obtain an invariant subspace $\cS^\times$ which yields a minimal factorization of $f$. By Proposition \ref{p:polyfac}, $f$ is not an atom, so Theorem \ref{t:irrpoly} holds.
\end{rem}

\subsection{A factorization result with missing variables}

Let $\uly=(y_1,\dots,y_h)$ be another tuple of freely noncommuting variables. For every $n\in\N$ and $1\le j\le h$ let $\gY_j^{(n)}=(\gy_{j\imath\jmath})_{\imath\jmath}$ be an $n\times n$ generic matrix. For every $f\in\pxy$ we have $\det f(\gX^{(n)},\gY^{(n)})\in \kk[\ulxi,\ulyi]$.

\begin{prop}\label{p:an}
Let $f\in\pxy$ and $f(0)=1$. If $\det f(\gX^{(n)},\gY^{(n)})$ is independent of $\gY^{(n)}$ for every $n\in\N$, then $f\in\px$.
\end{prop}

\begin{proof}
Without loss of generality let $f$ be an atom. Let $f^{-1}=1+c^\ti L^{-1}b$ be a minimal FM-realization, where
$$b=\sum_{i=1}^g b_ix_i+\sum_{j=1}^h\tilde{b}_jy_j,\qquad L=I-\sum_{i=1}^gA_ix_i-\sum_{j=1}^h\tilde{A}_j y_j$$
for $c,b_i,\tilde{b}_j\in\kk^d$ and $A_i,\tilde{A}_j\in\mat{d}$. Since $f=1-c^\ti (L+bc^\ti)^{-1}b$, it suffices to show that $\tilde{A}_j=0$ and $\tilde{b}_j=0$ for $j=1,\dots,h$.

By the assumption we have
$$\det L(\gX^{(n)},\gY^{(n)})=\det f(\gX^{(n)},\gY^{(n)})
=\det f(\gX^{(n)},0)=\det L(\gX^{(n)},0)$$
for all $n\in\N$. By \cite[Proposition 3.3]{KV}, $\tilde{A}_1,\dots,\tilde{A}_h$ generate a nilpotent ideal in the $\kk$-algebra generated by $A_1,\dots,A_g,\tilde{A}_1,\dots,\tilde{A}_h$. Since $f$ is an atom, $L$ is irreducible by Lemma \ref{l:flip1} and Proposition \ref{p:conv}. Then $A_1,\dots,A_g,\tilde{A}_1,\dots,\tilde{A}_h$ generate $\mat{d}$, which is a simple algebra and therefore $\tilde{A}_j=0$ for all $j$.

To prove $\tilde{b}_j=0$ for a fixed $1\le j\le h$ it therefore suffices to show that $c^\ti w(A)\tilde{b}_j=0$ for all $w\in\mx$, which we prove by induction on $|w|$. Note that $$A_1-b_1c^\ti,\dots,A_g-b_gc^\ti,\tilde{b}_1c^\ti,\dots, \tilde{b}_h c^\ti$$
are jointly nilpotent because $1-c^\ti(L+bc^\ti)^{-1}b$ is a minimal realization of a noncommutative polynomial. Firstly, $c^\ti \tilde{b}_j=\tr(\tilde{b_j}c^\ti)=0$ because $\tilde{b_j}c^\ti$ is nilpotent. Now suppose that $c^\ti w(A)\tilde{b}_j=0$ holds for all $w$ with $|w|\le \ell$; then also $c^\ti w(A-bc^\ti)\tilde{b}_j=0$ for all $w$ with $|w|\le \ell$. Hence for every $i=1,\dots,g$ and $w\in\ell$ we have
\begin{align*}
c^\ti (x_iw)(A)\tilde{b}_j
& =c^\ti (A_i-b_ic^\ti)w(A)\tilde{b}_j \\
& =c^\ti (A_i-b_ic^\ti)w(A-bc^\ti)\tilde{b}_j \\
& =\tr\left((x_iw)(A-bc^\ti)\tilde{b}_jc^\ti\right) \\
& =0
\end{align*}
because $(x_iw)(A-bc^\ti)\tilde{b}_jc^\ti$ is nilpotent.
\end{proof}

\begin{cor}\label{c:an}
Let $f\in\pxy$ and $f(0)\neq0$. If for large $n$, the polynomial $\det f(\gX^{(n)},\gY^{(n)})$ has a factor independent of $\gY^{(n)}$, then $f$ has a factor in $\px$.
\end{cor}

\begin{proof}
Immediate consequence of Theorem \ref{t:irrpoly} and Proposition \ref{p:an}.
\end{proof}

\section{Algorithms}

In this section we present algorithms based off the results presented above. In Subsection \ref{ss64} we present a simple algorithm for comparing free loci, and in Subsection 
\ref{ssec:factor} we give an algorithm for factorization of noncommutative polynomials.

\subsection{Comparing polynomial free loci}\label{ss64}

While Theorem \ref{t:irrpoly} characterizes the equality of free loci, stable associativity seems evasive to check directly. Thus we now describe a procedure that is more practical for checking equality $\fl(f_1)=\fl(f_2)$ for  noncommutative polynomials $f_1$ and $f_2$.

Let $f_1,f_2\in\px$ be such that $f_i(0)\neq0$. Then we can test for $\fl(f_1)=\fl(f_2)$ as follows. First we compute minimal FM-realizations for $f_1^{-1}$ and $f_2^{-1}$, which can be effectively done using algorithms based on linear algebra \cite{BGM}. Let $L_1$ and $L_2$ be the monic pencils appearing in these realizations. Since $\fl(f_1)=\fl(f_2)$ is equivalent to $\fl(L_1)=\fl(L_2)$, it then suffices to compare irreducible blocks in the invariant subspace decomposition for the coefficients of $L_1$ and $L_2$, which can be done using probabilistic algorithms with polynomial complexity \cite{Ebe,CIW}.

\begin{exa}
Let
\begin{align*}
f_1 &= 1+x_1+x_2+x_1^2x_2, \\
f_2 &= 1+x_1+x_2+x_1x_2x_1, \\
f_3 &= 1+x_1+x_2+x_2x_1^2.
\end{align*}
Since every affine linear polynomial clearly has an FM-realization of size 1 (note that $b$ in \eqref{e:fm} is linear), we can use the constructions of FM-realizations associated with the sum, product and inverse \cite[Section 4]{BGM} to build FM-realizations for $f_i^{-1}$ of size $5$. After applying the minimization algorithm we obtain minimal realizations $f_i^{-1}=1+c^\ti L_i^{-1}b_i$, where $c^\ti=(1,0,0)$ and
\begin{alignat*}{3}
L_1 &=\begin{pmatrix}
1+x_1+x_2 & x_1 & 0 \\
0 & 1 & -x_1 \\
-x_2 & 0 & 1
\end{pmatrix},\qquad
b_1 &= \begin{pmatrix}-x_1 - x_2 \\ 0 \\ x_2\end{pmatrix}, \\
L_2 &=\begin{pmatrix}
1+x_1+x_2 & x_1 & 0 \\
0 & 1 & -x_2 \\
-x_1 & 0 & 1
\end{pmatrix},\qquad
b_2 &= \begin{pmatrix}-x_1 - x_2 \\ 0 \\ x_1\end{pmatrix}, \\
L_3 &=\begin{pmatrix}
1+x_1+x_2 & x_2 & 0 \\
0 & 1 & -x_1 \\
-x_1 & 0 & 1
\end{pmatrix},\qquad
b_3 &= \begin{pmatrix}-x_1 - x_2 \\ 0 \\ x_1\end{pmatrix}.
\end{alignat*}
It is easy to check that the coefficients of $L_1$ generate $\mat{3}$, so $L_1$ is irreducible. Next we consider two homogeneous linear systems $PL_1=L_2P$ and $PL_1=L_3P$, where $P$ is a $3\times3$ matrix of scalar indeterminates. While the second system admits only the trivial solution $P=0$, the first system has a one-dimensional solution space which intersects $\GL_3(\kk)$. Therefore $L_1$ and $L_3$ are similar but $L_2$ is not similar to $L_1$. Consequently,
$$\fl(f_1)=\fl(f_3)\neq\fl(f_2).$$
Alternatively, one can use a computer algebra system to check $\fl_n(f_1)\setminus \fl_n(f_2)\neq\emptyset$ for a fixed $n$. For example, we have
$$\left(
\begin{pmatrix}1 & -1 \\ -1 & 0\end{pmatrix},\begin{pmatrix}1 & 1 \\ 1 & 0\end{pmatrix}
\right)\in \fl_2(f_1)\setminus\fl_2(f_2).$$
\end{exa}

\subsection{Factorization via state space realizations}\label{ssec:factor}

We now describe an algorithm for factorization of noncommutative polynomials based on their FM-realizations (cf. \cite{Scr}). As with comparing free loci, the only computational expenses of the algorithm arise from construction and minimization of FM-realizations, and finding an invariant subspace of a tuple of matrices.

Given $f\in\px$ with $f(0)=1$ we first find a minimal realization $f^{-1}=1+c^\ti L^{-1}b$; note that $\det f(\gX^{(n)})=\det L(\gX^{(n)})$ by Lemma \ref{l:flip1}.

\begin{enumerate}
\item If the coefficients of $L$ do not admit a non-trivial invariant subspace, $L$ is an irreducible pencil, so $\det f(\gX^{(n)})=\det L(\gX^{(n)})$ is irreducible for large $n$ and hence $f$ is an atom.

\item If the coefficients of $L$ admit a non-trivial invariant subspace, then we find an irreducible one, $\cS$. With respect to it we have
\begin{equation}\label{e:70}
f^{-1}=1+
\begin{pmatrix}c_1^\ti & c_2^\ti \end{pmatrix}
\begin{pmatrix}
L_1 & \star \\ 0 & L_2
\end{pmatrix}^{-1}
\begin{pmatrix}b_1 \\ b_2 \end{pmatrix}
\end{equation}
for monic pencils $L_1$ and $L_2$. Following the proof of Theorem \ref{t:pert} (Subsection \ref{sss:pfPert}), their coefficients are not jointly nilpotent, so $f$ factors. Moreover, since $\cS$ is irreducible, the same reasoning as in Subsection \ref{sss:pfPert} implies that 
the realization \eqref{e:70} necessarily yields a minimal factorization $f^{-1}=f_1^{-1}f_2^{-1}$ and from \eqref{e:minfac} we read off $f_i^{-1}=1+c_i^\ti L_i^{-1}b_i$. Therefore $f=f_2f_1$ is a polynomial factorization by Proposition \ref{p:polyfac} and $f_1$ is an atom.
\end{enumerate}

Hence we obtain a factorization of $f$ into atomic factors by repeating (2) until (1) applies. Note that the obtained factorization depends on the choice of irreducible invariant subspaces in (2).

\begin{exa}
Let
$$f=1+\tfrac32 x_1+\tfrac12 x_2+\tfrac12(x_1^2+x_1x_2+x_2x_1)+\tfrac12 x_1x_2x_1.$$
Then $f^{-1}$ admits a minimal realization
\begin{equation}\label{e:71}
1+\begin{pmatrix}1 & 0 & 0 \end{pmatrix}
\begin{pmatrix}
1+\tfrac32 x_1+\tfrac12 x_2 & x_1 & -x_1+\tfrac12 x_2 \\
-\tfrac32 x_1-\tfrac12 x_2 & 1 & -\tfrac12 x_2 \\
-x_1 & 0 & 1
\end{pmatrix}^{-1}
\begin{pmatrix}-\tfrac32 x_1-\tfrac12 x_2 \\ \tfrac32 x_1+\tfrac12 x_2 \\ x_1 \end{pmatrix}.
\end{equation}
Let $L$ be the monic pencil in \eqref{e:71}. The coefficients of $L$ have a common eigenvector $(-1,\tfrac32,1)^\ti$ and one can check that (2) yields
$$f= (1+\tfrac12 x_1+\tfrac12 x_2+\tfrac12 x_1x_2)(1+x_1).$$
However, the linear span of $(1,0,1)^\ti$ and $(-1,1,0)^\ti$ is also an irreducible invariant subspace for the coefficients of $L$, and in this case (2) results in
$$f=(1+x_1)(1+\tfrac12 x_1+\tfrac12 x_2+\tfrac12 x_2x_1).$$
Let $L^\times$ be the monic pencil appearing in the inverse of the realization \eqref{e:71}. Its coefficients generate the algebra of strictly upper triangular $3\times 3$ matrices, so they admit exactly two non-trivial invariant subspaces, namely $\spa\{(1,0,0)^\ti\}$ and $\spa\{(1,0,0)^\ti,(0,1,0)^\ti\}$. From the perspective of Theorem \ref{t:pert}, the first one is complementary to $\spa\{(-1,\tfrac32,1)^\ti\}$ and the second one is complementary to $\spa\{(1,0,1)^\ti,(-1,1,0)^\ti\}$.

\end{exa}

\section{Smooth points on a free locus}\label{ss51}

Let $L$ be a monic pencil. In this section we study the relation between smooth points of the free locus $\fl(L)$ and one-dimensional kernels of evaluations of the pencil $L$. Let us define
\begin{alignat*}{3}
\fl^1(L) &=\bigcup_{n\in\N}\fl^1_n(L),\qquad
&\fl^1_n(L) &=\left\{X\in\fl_n(L)\colon \dim\ker L(X)=1\right\}, \\
\fl^\one(L) &=\bigcup_{n\in\N}\fl^\one_n(L),\qquad
&\fl^\one_n(L) &=\left\{X\in\fl_n(L)\colon \dim\ker L(X)^2=1\right\}.
\end{alignat*}
That is, $X\in \fl^1(L)$ if the geometric multiplicity of the zero eigenvalue in $L(X)$ is $1$ and $X\in \fl^\one(L)$ if the algebraic multiplicity of the zero eigenvalue in $L(X)$ is $1$. Note that $\fl^\one_n(L)\subseteq\fl^1_n(L)$ are Zariski open subsets of $\fl_n(L)$ for every $n\in\N$. However, the set $\fl^1_n(L)$ can be empty for a fixed $n$ even if $L$ is an irreducible pencil; see Laffey's counterexample to Kippenhahn's conjecture \cite{Laf}.

Recall that $E_{\imath\jmath}\in\mat{n}$ denote the standard matrix units.

\begin{lem}\label{l:jac}
Let $L=I-\sum_jA_jx_j$ and $X\in\fl_n(L)$.
\begin{enumerate}
	\item $X\notin \fl^1_n(L)$ if and only if $\adj L(X)=0$.
	\item If $\det L(\gX^{(n)})$ is a minimum degree defining polynomial for $\fl_n(L)$, then $\fl_n(L)$ is singular at $X$ if and only if $\tr(\adj L(X)(A_j\otimes E_{\imath\jmath}))=0$
	for all $1\le j\le g$ and $1\le \imath,\jmath\le n$.
	\item $X\notin \fl^\one_n(L)$ if and only if $\tr(\adj L(X)(\sum_j A_j\otimes X_j))=0$.
\end{enumerate}
\end{lem}

\begin{proof}
Firstly, (1) is clear by the definition of the adjugate. Denote $f=\det L(\gX^{(n)})$ and $p_X=\det L(tX)$ for $X\in\mat{n}^g$. Then Jacobi's formula \cite[Theorem 8.3.1]{MN} for the derivative of a determinant implies
\begin{align*}
	\nabla f &= -\bigg(\tr\left(\adj L\left(\gX^{(n)}\right)(A_j\otimes E_{\imath\jmath})\right)\bigg)_{j,\imath,\jmath}, \\
	\frac{{\rm d} p_X}{{\rm d} t} &= -\tr\left(\adj L(tX)\left( \sum_j A_j\otimes X_j \right) \right).
\end{align*}
Now (2) and (3) follow because $\fl_n(L)$ is singular at $X$ if and only if $(\nabla f)(X)=0$, and $X\notin \fl^\one_n(L)$ if and only if $\frac{{\rm d} p_X}{{\rm d} t}(1)=0$.
\end{proof}

\begin{thm}\label{t:fl-min}
If $L$ is an FL-minimal pencil, then
$$\emptyset\neq\fl^\one_n(L)\subseteq \left\{\text{smooth points of } \fl_n(L) \right\} \subseteq \fl^1_n(L)$$
for large $n$.
\end{thm}
\begin{proof}
By Remark \ref{r:rad}, $\det L(\gX^{(n)})$ is a minimum degree defining polynomial for $\fl_n(L)$ for large $n$. Hence the inclusions hold by Lemma \ref{l:jac}. Note that
$$\frac{{\rm d}}{{\rm d} t} \det L(t\gX^{(n)})=-\tr \left(
\adj L(t\gX^{(n)})\left(\sum_j A_j\otimes \gX_j^{(n)} \right)
\right)$$
and hence
\begin{align*}
	\deg \det L(\gX^{(n)})
	&= \deg_t \det L(t\gX^{(n)}) \\
	&=1+ \deg_t \frac{{\rm d}}{{\rm d} t} \det L(t\gX^{(n)}) \\
	&=1+\deg_t \tr \left(
	\adj L(t\gX^{(n)})\left(\sum_j A_j\otimes \gX_j^{(n)} \right)
	\right) \\
	&= \deg \tr \left(
	\adj L(\gX^{(n)})\left(\sum_j A_j\otimes \gX_j^{(n)} \right)
	\right).
\end{align*}
Note that the constant terms of $\tr(\adj L(\gX^{(n)})(\sum_j A_j\otimes \gX_j^{(n)}))$ and $\det L(\gX^{(n)})$ equal 0 and 1, respectively. Since these two polynomials have the same degree, we conclude that $\tr(\adj L(\gX^{(n)})(\sum_j A_j\otimes \gX_j^{(n)}))$ is not a multiple of $\det L(\gX^{(n)})$. Therefore $\fl^\one_n(L)\neq\emptyset$ for large $n$ by Theorem \ref{t:main} and Lemma \ref{l:jac}(2).
\end{proof}

\begin{exa}
Let $L=I-A_1x_1-A_2x_2$ be as in Example \ref{ex0} and
$$X=\left( \begin{pmatrix}1 & 0 \\ 0 & 1 \end{pmatrix},\begin{pmatrix}0 & 0 \\ -1 & -1 \end{pmatrix}\right),
\qquad
Y=\left( \begin{pmatrix}2 & \frac12 \\ 0 & 1 \end{pmatrix},\begin{pmatrix}0 & 0 \\ 2 & 1 \end{pmatrix}\right).$$
One can check that $\fl_2(L)$ is singular at $X\in\fl^1_2(L)$ and smooth at $Y\in\fl_2(L)\setminus\fl^\one_2(L)$. Hence the inclusions in Proposition \ref{p:sing} are strict in general.
\end{exa}

The next proposition describes the behavior of smooth points when moving between different levels of a free locus.

\begin{prop}\label{p:sing}
Let $L$ be a monic pencil, $X\in\fl(L)$ and $Y\in\mat{n}^g$.
\begin{enumerate}
	\item If $Y\in\fl(L)$, then $\fl(L)$ is singular at $X\oplus Y$.
	\item If $Y\notin\fl(L)$ and $n$ is large enough, then $\fl(L)$ is smooth at $X\oplus Y$ if and only if $\fl(L)$ is smooth at $X$.
\end{enumerate}
\end{prop}

\begin{proof}
Since the statement is about the free locus and not $L$ directly, we can assume that $L$ is FL-minimal.

(1) If $X,Y\in\fl(L)$, then $X\oplus Y\notin \fl^1(L)$ and hence $X\oplus Y$ is not a smooth point of $\fl(L)$ by Theorem \ref{t:fl-min}.

(2) Observe that $\adj(M_1\oplus M_2)=(\det M_1 \adj M_2)\oplus (\det M_2\adj M_2)$ for arbitrary $M_1,M_2$. Hence
$$\adj L(X\oplus Y)= (\det L(Y)\adj L(X))\oplus 0$$
and the equivalence follows by Lemma \ref{l:jac}(2).
\end{proof}

The quasi-affine variety $\fl_n^1(L)$ comes equipped with a natural line bundle: to each $X\in\fl^1(L)$ we assign the line $\ker L(X)$. Define $\pi:\kk^{dn}\to\kk^d$ by setting $\pi(v)=u_1$ for $v=\sum_{i=1}^n u_i\otimes e_i\in \kk^d\otimes \kk^n$. For a monic pencil $L$ define
$$\hair(L)=\bigcup_{X\in \fl^1(L)}\pi\left( \ker L(X) \right)\subseteq \kk^d.$$
Since $\fl^1_n(L)$ is closed under $\gl{n}$-conjugation, we have
$$\hair(L)=\bigcup_{X\in \fl^1(L)}\left\{u_i\in\kk^d\colon \sum_i u_i\otimes e_i\in\ker L(X)\right\}.$$

\begin{prop}\label{p:hairspan}
If $L$ is irreducible of size $d$, then $\spa \hair(L)=\kk^d$.
\end{prop}

\begin{proof}
Let $L=I-\sum_jA_jx_j$ be irreducible and suppose $\spa \hair(L)\neq\kk^d$. Let $P\in \mat{d}$ be a projection onto $\spa \hair(L)$ and $L'=I-\sum_jA_jPx_j$. Let $X\in \fl^1(L)$ be arbitrary. If $L(X)v=0$, then $v=(P\otimes I)v$ and
$$L'(X)v=\left(I-\sum_jA_jP\otimes X_j\right)v=v-\sum_j\left(A_jP\otimes X_j\right)v=v-\sum_j\left(A_j\otimes X_j\right)v=0,$$
so $X\in \fl(L')$. Since $\fl_n^1(L)$ contains all the smooth points of $\fl_n(L)$ for large $n\in\N$ by Theorem \ref{t:fl-min} an is therefore dense in $\fl_n(L)$, we conclude that $\fl(L)\subseteq\fl(L')$. By \cite[Theorem 3.6]{KV} there exists a surjective homomorphism $\cA'\to \mat{d}$ given by $A_jP\mapsto A_j$, where $\cA'$ is the $\kk$-algebra generated by $A_1P,\dots,A_gP$. But $\cA'\subsetneq \mat{d}$, a contradiction.
\end{proof}

The above results can be also applied to free loci of more general matrices of noncommutative polynomials. For example, we obtain the following.

\begin{cor}\label{c:hair}
Let $f\in\opm_d(\px)$ be an atom with $f(\kk^g)\cap \GL_d(\kk)\neq\emptyset$. Then there exists $X\in\fl(f)$ such that $\dim\ker f(X)=1$.
\end{cor}

\begin{proof}
By Remark \ref{r:shift} we can assume $f(0)=I$. By Lemma \ref{l:sa}, $f$ is stably associated to an irreducible monic pencil $L$, so there exists $X\in\fl(L)$ satisfying $\dim\ker L(X)=1$ by Theorem \ref{t:fl-min}. By the definition of stable associativity we then have $\dim\ker f(X)=1$.
\end{proof}

\section{Applications to real algebraic geometry}\label{ss52}

In this section we present two applications
of our results to real and convex algebraic geometry.
In Corollary \ref{c:pd-min}
we prove a density result for points $X$ on the 
boundary of a free spectrahedron determined
by a  hermitian pencil $L$, where the kernel of $L(X)$ is one-dimensional. As a consequence we obtain Corollary \ref{c:rand}, which improves upon the main result of \cite{HKN}.

\subsection{Boundaries of free spectrahedra}

Let $\herm{n}$ denote the $\R$-space of $n\times n$ hermitian matrices. A monic pencil $L=I-\sum_jA_jx_j$ is {\bf hermitian} if $A_j\in\herm{n}$ for $1\le j \le g$. Its {\bf free spectrahedron} (also called {\bf LMI domain}) \cite{HKM} is the set
$$\cD(L)=\bigcup_{n\in\N}\cD_n(L),\qquad \cD_n(L)=\left\{X\in\herm{n}^g\colon L(X)\succeq0\right\}.$$
Also denote
\begin{alignat*}{3}
\flh(L)&=\bigcup_{n\in\N}\flh_n(L),\qquad &\flh_n(L)&=\fl_n(L)\cap\herm{n}^g, \\
\partial\cD(L)&=\bigcup_{n\in\N}\partial\cD_n(L),\qquad &\partial\cD_n(L)&=\cD_n(L)\cap\flh_n(L), \\
\partial^1\cD(L)&=\bigcup_{n\in\N}\partial^1\cD_n(L),\qquad &\partial^1\cD_n(L)&=\partial\cD_n(L)\cap\fl_n^1(L).
\end{alignat*}
The set $\flh(L)$ is the {\bf free real locus} of $L$. A hermitian monic pencil $L$ is {\bf LMI-minimal} if it is of minimal size among all hermitian pencils $L'$ satisfying $\cD(L')=\cD(L)$. Note that if $L_1$ and $L_2$ are hermitian pencils, then $\fl(L_1)=\fl(L_2)$ implies $\cD(L_1)=\cD(L_2)$. Using Burnside's theorem and the hermitian structure of an LMI-minimal $L$ it is then easy to deduce that $L$ is unitarily equivalent to $L_1\oplus\cdots\oplus L_\ell$, where $L_k$ are pairwise non-similar irreducible hermitian pencils. In particular, every LMI-minimal pencil is also FL-minimal.

\begin{rem}\label{r:xxs}
Instead of hermitian monic pencils in hermitian variables $\ulx$ as above, one can also consider hermitian monic pencils in non-hermitian variables $\ulx$ and $\ulx^*$, i.e., pencils of the form $L=I-\sum_jA_jx_j-\sum_jA_j^*x_j^*$ for $A_j\in\matc{d}$, with evaluations
$$L(X)=I-\sum_jA_j\otimes X_j-\sum_jA_j^*\otimes X_j^*$$
for $X\in\matc{n}^g$. However, by introducing new hermitian variables $y_j=\frac12(x_j+x_j^*)$ and $z_j=\frac{1}{2i}(x_j-x_j^*)$ we observe that results about free real loci and LMI domains of pencils in hermitian variables, for instance those in \cite{KV}, readily translate into results about free real loci and LMI domains of pencils in non-hermitian variables. More concretely, let
$$L'=I-\sum_j(A_j+A_j^*)y_j-\sum_j i(A_j-A_j^*)z_j.$$
Involution-free properties are the same for $L$ and $L'$; for example, $L$ is irreducible if and only if $L'$ is irreducible, and $L(\gX^{(n)}+i\gY^{(n)},\gX^{(n)}-i\gY^{(n)})$ is irreducible if and only if $L'(\gX^{(n)},\gY^{(n)})$ is irreducible. Likewise, topological relations (e.g., density) among the sets
\begin{equation}\label{e:ser1}
\partial^1\cD(L)\subset\partial\cD(L)\subset\flh(L)\subset \fl(L)
\end{equation}
are the same as those among
\begin{equation}\label{e:ser2}
\partial^1\cD(L')\subset\partial\cD(L')\subset\flh(L')\subset \fl(L')
\end{equation}
because one passes between \eqref{e:ser1} and \eqref{e:ser2} via $\R$-linear transformations. Using these two observations it becomes clear that the following results (Lemma \ref{l:smooth}, Proposition \ref{p:pd-min}, and Corollaries \ref{c:pd-min} and \ref{c:rand}) also hold in the $(\ulx,\ulx^*)$ setup.
\end{rem}

\begin{lem}\label{l:smooth}
Let $L$ be an LMI-minimal hermitian pencil. Then $\partial^1\cD_n(L)$ are precisely the smooth points of $\partial\cD_n(L)$ for large $n$. 
\end{lem}

\begin{proof}
The polynomial $\det L(\gX^{(n)})$ is square-free for large $n$ by Remark \ref{r:rad}. Let us consider $\matc{n}^g=\herm{n}^g+i \herm{n}^g$ as the decomposition of the affine space $\matc{n}^g$ into its real and imaginary part. Since $L$ is hermitian, $\det L(\gX^{(n)})$ is a complex analytic polynomial with real coefficients. Let $F$ be the homogenization of $\det L(\gX^{(n)})$. Viewed as a real polynomial, $F$ is hyperbolic with respect to the direction of the homogenizing variable because $L$ is a monic and hermitian. Since $\det L(\gX^{(n)})$ is square-free, $F$ is also square-free. Thus it follows by \cite[Lemma 7]{Ren} that $\partial^1\cD_n(L)$ are precisely the smooth points of $\partial\cD_n(L)$ for large $n$. 
\end{proof}

Given an LMI-minimal pencil $L$, smooth points of the boundary of the free spectrahedron of $L$ are (at least for large sizes) characterized as the points where the kernel of $L$ attains minimal dimension. On the other hand, the points where $L$ has maximal kernel are (Euclidean) extreme points of the spectrahedron \cite{RG,EHKM,DDOSS,ANT,Kri}.

\begin{prop}\label{p:pd-min}
Let $L$ be an LMI-minimal hermitian pencil. Then $\partial\cD_n(L)$ is Zariski dense in $\fl_n(L)$ for large $n$.
\end{prop}

\begin{proof}
Let $L=L_1\oplus\cdots\oplus L_\ell$, where $L_k$ are pairwise non-similar irreducible hermitian pencils. Fix $1\le k\le \ell$. By the minimality of $L$ we have
\begin{equation}\label{e:31}
	\bigcap_{k'\neq k}\cD(L_{k'})\not\subseteq \cD(L_k).
\end{equation}
Since
$$\bigcup_{k=1}^\ell(\partial\cD(L_k)\cap \partial\cD(L))\subseteq \partial\cD(L)\subseteq \fl(L)=\bigcup_{k=1}^\ell\fl(L_k),$$
it suffices to prove that $\partial\cD_n(L_k)\cap \partial\cD_n(L)$ is Zariski dense in $\fl_n(L_k)$ for large $n$.

Since the LMI domain of a hermitian monic pencil is a convex set with nonempty interior, then by \eqref{e:31} for large $n$ there exists $X_0\in\herm{n}^g$ such that $L_k(X_0)\not\succeq0$ and $L_{k'}(X_0)\succ0$ for $k'\neq k$. Since this is an open condition in Euclidean topology, there exists $\ve>0$ such that for every $X\in B(X_0,\ve)$ we have $L_k(X)\not\succeq0$ and $L_{k'}(X)\succ0$ for $k'\neq k$, where $B(X_0,\ve)\subset\herm{n}^g$ is the closed ball about $X_0$ with radius $\ve$ in Euclidean norm. Let $\cC$ be the convex hull of the origin and $B(X_0,\ve)$. By convexity we have $\cC\subset \bigcap_{k'\neq k}\cD(L_{k'})$ and thus
$$\flh_n(L_k)\cap\cC\subseteq\partial\cD_n(L_k)\cap \partial\cD_n(L).$$
Observe that for every $X\in B(X_0,\ve)$ there exists $t\in(0,1)$ such that $\det L_k(tX)=0$ by the choice of $X_0$ and $\ve$. Therefore $\flh_n(L_k)\cap\cC\subseteq \herm{n}^g$ is a semialgebraic set of (real) dimension $gn^2-1$ by \cite[Theorem 2.8.8]{BCR}. Therefore its Zariski closure in $\matc{n}^g=\herm{n}^g+i \herm{n}^g$ is a hypersurface by \cite[Proposition 2.8.2]{BCR}. Since the latter is contained in $\fl_n(L_k)$, which is an irreducible hypersurface for large $n$, we conclude that $\flh_n(L_k)\cap\cC$ is Zariski dense in $\fl_n(L_k)$ and therefore $\partial\cD_n(L_k)\cap \partial\cD_n(L)$ is Zariski dense in $\fl_n(L_k)$ for large $n$.
\end{proof}

\begin{rem}\label{r:euc}
The essence of the last proof is that $\partial^1\cD_n(L)$ has a nonempty interior with respect to the Euclidean topology on $\flh_n(L)$ for large $n$ if $L$ is an irreducible hermitian pencil.
\end{rem}

The following statement is a spectrahedral version of the quantum Kippenhahn conjecture (cf. \cite[Corollary 5.7]{KV}).

\begin{cor}\label{c:pd-min}
Let $L$ be an LMI-minimal hermitian pencil. Then $\partial^1\cD_n(L)$ is Zariski dense in $\fl_n(L)$ for large $n$.
\end{cor}

\begin{proof}
For large $n$, $\fl^1_n(L)$ is Zariski dense and open in $\fl_n(L)$ by Theorem \ref{t:fl-min} and $\partial\cD_n(L)$ is Zariski dense in $\fl_n(L)$ by Proposition \ref{p:pd-min}. Therefore $\partial^1\cD_n(L)$ is Zariski dense in $\fl_n(L)$ for large $n$.
\end{proof}

\begin{rem}
Let us consider the real symmetric setup, where the coefficients of a monic pencil $L$ are real symmetric matrices and we are only interested in evaluations of $L$ on tuples of real symmetric matrices. Then the analog of Corollary \ref{c:pd-min} fails in general. For example, let $L$ be a monic symmetric pencil of size $4d$ whose coefficients generate the algebra of $d\times d$ matrices over quaternions; then $L$ is irreducible over $\R$ and thus LMI-minimal as a pencil over $\R$. However, $L$ is unitarily equivalent to $L'\oplus L'$ for an irreducible hermitian pencil $L'$, so $\fl^1(L)=\emptyset$. On the positive side, the real version of quantum Kippenhahn's conjecture holds: if $L$ is a symmetric irreducible pencil over $\R$, then by \cite[Corollary 5.8]{KV} there exists a tuple of symmetric matrices $X$ such that $\dim\ker L(X)=2$.
\end{rem}

\subsection{Randstellensatz}

An important result in free real algebraic geometry is the Randstellensatz \cite[Theorem 1.1]{HKN} which describes noncommutative polynomials defining a given LMI domain and its boundary. It holds for monic pencils $L$ satisfying the ``zero determining property'' \cite[Subsection 5.2]{HKN}. Without going into technical details we assert that every LMI-minimal hermitian pencil $L$ satisfies the zero determining property because $\det L(\gX^{(n)})$ is a minimum degree defining polynomial for $\fl_n(L)$ by Remark \ref{r:rad} and thus also for the Zariski closure of $\partial^1\cD_n(L)$ in $\herm{n}^g$ (which equals $\flh_n(L)$) by Corollary \ref{c:pd-min}. Thus we obtain the following improvement of \cite[Theorem 1.1]{HKN}.

\begin{cor}\label{c:rand}
Let $L$ be an LMI-minimal hermitian pencil of size $d$ and $f\in \opm_d(\pxc)$. Then
$$f|_{\cD(L)}\succeq0\qquad \text{and} \qquad \ker L(X)\subseteq\ker f(X)\quad \forall X\in\cD(L)$$
if and only if
$$f=L\left(\sum_i q_i^*q_i\right)L+\sum_j (r_j L+C_j)^*L(r_j L+C_j)$$
for $q_i\in\pxc^d$, $r_j\in\opm_d(\pxc)$ and $C_j\in\matc{d}$ satisfying $C_jL=LC_j$.
\end{cor}


\end{document}